\documentclass[12pt,oneside]{amsart}
\usepackage{etex}
\usepackage{enumitem}
\usepackage{hyperref}
\usepackage{xcolor}
\hypersetup{
    colorlinks,
    linkcolor={blue!80!black},
    citecolor={blue!50!black},
    urlcolor={blue!80!black}
}

\usepackage{amsfonts,amssymb,amscd,amsmath,latexsym,amsbsy, setspace}
\usepackage{amsfonts,amssymb,graphicx,amsthm,mathtools}
\usepackage{fullpage}
\usepackage[all,cmtip,curve, knot,frame]{xy} 
\usepackage{amssymb, amsmath}
\usepackage{amsfonts}
\usepackage{latexsym}
\usepackage{epstopdf}
\usepackage{tikz}
\usepgflibrary{shapes.geometric}
\usetikzlibrary{matrix}
\usepackage{pgfplots}
\usepackage{braids}

\newtheorem{theorem}{Theorem}[section]
\newtheorem{lemma}[theorem]{Lemma}

\newtheorem{proposition}[theorem]{Proposition}
\newtheorem{corollary}[theorem]{Corollary}
\theoremstyle{definition}
\newtheorem{definition-proposition}[theorem]{Definition-Proposition}
\newtheorem{definition}[theorem]{Definition}

\newtheorem{remark}[theorem]{Remark}

\newcommand{\RV}{\mf R_G}
\newcommand{\oa}{\mathfrak{F}_\mathcal{A}}	
\newcommand{\sslash}{\mathbin{/\mkern-6mu/}}
\newcommand{\LFP}{\operatorname{\mathbf{Pr}}}
\newcommand{\ind}{\operatorname{ind}}
\newcommand{\comp}{\operatorname{comp}}

\newcommand{\mf}{\mathfrak}
\newcommand{\rightloop}{%
      \mathrel{\raisebox{.1em}{%
      \reflectbox{\rotatebox[origin=c]{-90}{$\circlearrowright$}}}}}

\newcommand{\Vect}{\operatorname{Vect}}

\newcommand{\End}{\operatorname{End}}
\newcommand{\Hom}{\operatorname{Hom}}

\newcommand{\Disk}{\operatorname{Disk}}
\newcommand{\Mfld}{\operatorname{Mfld}}

\newcommand{\Rep}{\operatorname{Rep\,}}
\newcommand{\Repq}{\operatorname{Rep_q}}
\newcommand{\colim}[1]{\underset{#1}{\operatorname{colim}}}

\newcommand{\tr}{\text{tr}\,}

\newcommand{\K}{\mathbf{k}}

\newcommand{\cA}{\mathcal{A}}
\newcommand{\cB}{\mathcal{B}}
\newcommand{\cC}{\mathcal{C}}
\newcommand{\cD}{\mathcal{D}}

\newcommand{\cO}{\mathcal{O}}

\newcommand{\cZ}{\mathcal{Z}}

\newcommand{\cM}{\mathcal{M}}

\newcommand{\cN}{\mathcal{N}}

\newcommand{\QCoh}{\operatorname{QCoh}}

\newcommand{\Rex}{\operatorname{\mathbf{Rex}}}

\newcommand{\cE}{\mathcal{E}}
\newcommand{\dS}{\Big/\hspace{-5pt}\Big/}

\newcommand{\ot}{\otimes}
\newcommand{\bt}{\boxtimes}
\newcommand{\id}{\operatorname{id}}

\newcommand{\RR}{\mathbb{R}}
\newcommand{\CC}{\mathbb{C}}

\newcommand{\ZZ}{\mathbb{Z}}

\newcommand{\ev}{\operatorname{ev}}

\newcommand{\Mat}{\operatorname{Mat}}
\newcommand{\ltr}{\triangleleft}

\newcommand{\rtr}{\triangleright}

\newcommand{\act}{\operatorname{act}}

\newcommand{\uch}{\underline{\operatorname{Ch}}}

\newcommand{\modu}{\operatorname{-mod}}
\newcommand{\bimodu}{\operatorname{-mod-}}
\newcommand{\modul}{\operatorname{mod}}

\newcommand{\un}{\mathbf{1}}

\usetikzlibrary{arrows}
\tikzstyle cross=[preaction={draw=white, -, line width=4pt}, thick]
\tikzstyle normal=[thick]
\tikzstyle chord=[densely dotted, thick]
\tikzstyle zero=[ultra thick, gray]
\tikzstyle zerocross=[preaction={draw=white, -, line width=4pt}, ultra thick, gray]
\tikzstyle point=[draw,circle,inner sep=1,fill=black]
\tikzstyle petitpoint=[draw,circle,inner sep=0.3,fill=black]

\newcommand{\negative}[3][-]{\draw[normal,#1] (#2+1,-#3).. controls (#2+1,-#3-0.3) and (#2,-#3-0.7)..(#2,-#3-1); \draw[cross,#1] (#2,-#3).. controls (#2,-#3-0.3) and (#2+1,-#3-0.7)..(#2+1,-#3-1);}
\newcommand{\positive}[3][-]{ \draw[normal,#1] (#2,-#3).. controls (#2,-#3-0.3) and (#2+1,-#3-0.7)..(#2+1,-#3-1);\draw[cross,#1] (#2+1,-#3).. controls (#2+1,-#3-0.3) and (#2,-#3-0.7)..(#2,-#3-1);}
\newcommand{\lmove}[3][-]{\draw[normal,#1] (#2+1,-#3).. controls (#2+1,-#3-0.3) and (#2,-#3-0.7)..(#2,-#3-1);}
\newcommand{\rmove}[3][-]{\draw[cross,#1] (#2,-#3).. controls (#2,-#3-0.3) and (#2+1,-#3-0.7)..(#2+1,-#3-1);}
\newcommand{\coupon}[4]{\draw[normal] (#1-0.2,-#2) rectangle (#1+#3+0.2,-#2-1);\node at (#1+#3*0.5,-#2-0.5) {#4};}

\newcommand{\up}[3][-]{ 
\draw[normal,#1] (#2,-#3).. controls (#2,-#3-0.3) and (#2+0.2,-#3-0.5).. (#2+0.5,-#3-0.5);
\draw[normal] (#2+0.5,-#3-0.5).. controls (#2+0.8,-#3-0.5) and (#2+1,-#3-0.3).. (#2+1,-#3);}
\newcommand{\ap}[3][-]{
\draw[normal,#1] (#2,-#3-1).. controls (#2,-#3-0.7) and (#2+0.2,-#3-0.5).. (#2+0.5,-#3-0.5);
\draw[normal] (#2+0.5,-#3-0.5).. controls (#2+0.8,-#3-0.5) and (#2+1,-#3-0.7).. (#2+1,-#3-1);}

\newcommand{\straight}[3][-]{\draw[normal,#1] (#2,-#3) -- (#2,-#3-1);}

\newcommand{\tik}[1]{\begin{tikzpicture}[baseline=(current bounding box.center), scale=0.8] #1 \end{tikzpicture} }

\def\HH{\hbox{${\mathcal H}$\kern-5.2pt${\mathcal H}$}}

\numberwithin{equation}{section}
\begin{document}
\title[Quantum character varieties]{Quantum character varieties and braided module categories}
\author{David Ben-Zvi}\address{Department of Mathematics\\University
  of Texas\\Austin, TX 78712-0257} \email{benzvi@math.utexas.edu}
\author{Adrien Brochier} \address{MPIM, Bonn}\email{abrochier@mpim-bonn.mpg.de}
\author{David Jordan} \address{School of Mathematics\\University of Edinburgh\\ Edinburgh, UK}\email{D.Jordan@ed.ac.uk}

\maketitle
\begin{abstract} 
We compute quantum character varieties of arbitrary closed surfaces with boundaries and marked points.  These are categorical invariants $\int_S\cA$ of a surface $S$, determined by the choice of a braided tensor category $\cA$, and computed via factorization homology.

We identify the algebraic data governing marked points and boundary components with the notion of a {\em braided module category} for $\cA$, and we describe braided module categories with a generator in terms of certain explicit algebra homomorphisms called {\em quantum moment maps}. We then show that the quantum character variety of a decorated surface is obtained from that of the corresponding punctured surface as a quantum Hamiltonian reduction.

Characters of braided $\cA$-modules are objects of the torus category $\int_{T^2}\cA$. We initiate a theory of character sheaves for quantum groups by identifying the torus integral of $\cA=\Repq  G$ with the category $\cD_q(G/G)\modu$ of equivariant quantum $\cD$-modules. When $G=GL_n$, we relate the mirabolic version of this category
to the representations of the spherical double affine Hecke algebra (DAHA) $\mathbb{SH}_{q,t}$.  
\end{abstract}

\tableofcontents

\section{Introduction}
Let $S$ denote a topological surface and $G$ a reductive group. The $G$-{\em character stack} $\uch_G(S)$ of $S$ is the moduli space of $G$-local systems on $S$, the quotient of the affine scheme of representations of the fundamental group of $S$ into $G$ by the conjugation action of $G$. Character stacks -- and their variants associated to surfaces with marked points or other decorations, which we collectively refer to as character varieties -- play a central role in geometry, representation theory and physics. A crucial feature of character stacks is their local nature --- they are obtained from gluing stacks of local systems on patches of $S$. As a result they provide a natural source of topological field theories: numbers, such point counts/Euler characteristics of character varieties appear in two-dimensional field theory; vector spaces such as sections of line bundles on character varieties appear in three-dimensional field theory; and finally, categories of sheaves on character varieties appear naturally in four-dimensional field theory.

We will be concerned with the four-dimensional setting, accessing character varieties through their categories of coherent sheaves, as they appear in the Betti form of the Geometric Langlands program and in twisted $d=4$ $\cN=4$ super-Yang-Mills theory, following the work of Kapustin-Witten~\cite{Ben-Zvi2016a}. Our goal is to construct and describe {\em quantum character varieties} -- $q$-deformations of these categories, which quantize the Goldman symplectic structure on the character stack (associated to a choice of invariant form on $\mathfrak g$). Moreover we endow these quantum character varieties with all of the structures expected from their origin in topological field theory, and develop their study as a natural setting for a variety of constructions in quantum algebra. 

In \cite{Ben-Zvi2015}, we initiated the construction of quantum character varieties via factorization homology of braided (or balanced) tensor categories $\cA$ on topological surfaces: this produces category-valued invariants of framed (or oriented) topological surfaces, with the desired strong functoriality and locality properties.
Starting from the braided tensor category $\cA=\Repq  G$ of integrable representations of the corresponding quantum group yields the desired functorial quantizations.  We computed these invariants for unmarked punctured surfaces, in terms of certain explicitly presented, and in many cases well-known, quantum algebras, which were constructed as certain twisted tensor products of the so-called ``reflection equation algebra" $\oa\in \cA$.

In the present paper, we extend this framework to the setting of closed surfaces, as well as to surfaces with marked points and boundaries. In brief, our main results are as follows:

\begin{itemize}[leftmargin=10pt]
	\item[$\bullet$] The possible markings of points (codimension two defects) in the topological field theory defined by $\cA$ are given by module categories over the monoidal category $\int_{Ann}\cA$. In~\cite{Ben-Zvi2015} we identified the underlying category with modules $\oa\modu_\cA$ for the {\em reflection equation algebra} of $\cA$. In Section~\ref{annulus-category-section} we explicitly identify the new induced monoidal structure on this category, the {\em field-goal tensor product}.
	
\item[$\bullet$] We show in Theorem \ref{thm:e2mod} that codimension two defects ($\int_{Ann}\cA$-modules) are identified with \emph{braided module categories} over $\cA$ (in the sense of~\cite{Brochier2013,Brochier2012,Enriquez2008}), in the same way that the unmarked disc is assigned $\cA$ (see Section~\ref{sec:factHom} for an introduction to factorization homology of marked surfaces). There are many natural examples of braided module categories (see below), including ones corresponding to versions of character varieties with parabolic structures, fixed conjugacy classes, or other boundary conditions (codimension one defects wrapping a circle).
They play the role for the 4d Kapustin-Witten (Betti Geometric Langlands) TFT that integrable representations of the loop group play for the 3d Chern-Simons (Witten-Reshetikhin-Turaev) theory.
	
\item[$\bullet$] In Theorem \ref{quantum-moment-map-theorem}, we identify braided module categories with a generator as modules for an algebra object $A_\cM\in \cA$ equipped with a ``quantum moment map", i.e. an algebra homomorphism, $\mu:\oa\to A_\cM$.  As we explain, $\mu$ is a quantum version of a group-valued moment map appearing in the classical setting \cite{Alekseev1998}.  	

\item[$\bullet$] We describe (in  Theorems~\ref{reconstruction}, Proposition~\ref{inducedactions}, and 
Corollary~\ref{annulus-corollary}) the result of gluing braided module categories with generators over their common braided $\cA$-action as a category of bimodules in $\cA$.
The quantum moment maps play a key role in defining the bimodule structure, and the resulting categories may be regarded as categorical quantum Hamiltonian reductions, along the respective quantum moment maps.

\item[$\bullet$] In particular, we compute the ``global functions" on general quantum character varieties: the endomorphisms of the quantum structure sheaf on a closed (or marked) surface are identified as the quantum Hamiltonian reduction of the algebra $A_{S^\circ}$ associated to a punctured surface along the corresponding quantum moment map.

\item[$\bullet$] The torus integral $\int_{T^2}\cA$ of a balanced tensor category $\cA$ is identified with 
the trace (or Hochschild homology) of the 2-category of braided $\cA$-modules, i.e., the natural receptacle for characters of braided modules. For $\cA=\Repq  G$ we identify the torus integral with the category $\cD_q(G/G)\modu$ of adjoint-equivariant quantum $\cD$-modules, thus providing an interpretation for these characters as quantum analogues of character sheaves.

\item[$\bullet$] For $G=GL_n$, the category $\cD_q(G/G)\modu$ has a ``mirabolic" version obtained by marking a single point in $T^2$ by the quantum ``Ruijsenaars-Schneider'' conjugacy class $\cN=A_t\modu_\cA$. We show that we recover Cherednik's spherical double affine Hecke algebra $\mathbb{SH}_{q,t}$, as the ``global functions" (endomorphisms of the quantum structure sheaf).  Hence, the global sections of any mirabolic quantum $\cD$-module carries a canonical action of the spherical double affine Hecke algebra.

\end{itemize}

\subsection{Braided module categories and quantum moment maps}
Factorization homology provides a general mechanism to construct invariants of $n$-manifolds starting from algebras $\cA$ over the $E_n$ (little $n$-disks) operad -- i.e., objects which carry operations labeled by inclusions of disks into a larger disk (see Section~\ref{sec:factHom} for a brief review). The invariant $\int_M \cA$ of an $n$-manifold is then defined as the universal recipient of maps from $\cA$ for every inclusion of a disk into $M$, and factoring through the operad structure.

There is a natural operadic notion of module $\cM$ for an $E_n$-algebra $\cA$, captured pictorially by placing the module at a marked point of a disk and allowing insertions of $\cA$ at disjoint disks. 
It is well-known (see e.g. \cite{Ayala2012,Ginot2014}) that the structure of $E_n$-module over an $E_n$-algebra $\cA$ on $\cM$ is equivalent to the structure of left module over the associative ($E_1$) ``universal enveloping algebra" $U(\cA)$, namely the factorization homology $U(\cA)=\int_{Ann}\cA$ of an annulus with coefficients in $\cA$. The latter category is equipped with an $E_1$ (monoidal) structure, coming from concatenation of annuli.  This is the structure used in the excision axiom on $Ann=S^{n-1}\times I$ (see Section~\ref{sec:BraidedModule} for more details) by which one computes factorization homology.

In~\cite{Ben-Zvi2015} we initiated the study of factorization homology of surfaces with coefficients in braided tensor categories, which are precisely the $E_2$-algebras in a certain 2-category $\cC=\LFP$ of linear categories with the Kelly-Deligne tensor product. (More precisely, braided tensor categories $\cA$ can be integrated over {\em framed} surfaces, while equipping $\cA$ with a {\em balanced} structure extends this integral to oriented surfaces.) In the same way, factorization homology of surfaces with marked points demands that for each marked point we give an $E_2$-module $\cM$ over our chosen braided tensor category $\cA$.

In Theorem~\ref{thm:e2mod}, we show that in the case $\cC=\LFP$, the notion of an $E_2$-module is equivalent to that of a ``braided module category", a concept introduced in \cite{Brochier2012,Brochier2013,Enriquez2008}, and closely related to the reflection equation algebra.  
We also introduce the notion of a ``balanced braided module category", which captures the structure of a $\Disk^2_{or}$-module (i.e. the oriented marked case), and we show in Theorem~\ref{thm:balancedModule} that, when $\cA$ itself is balanced, it endows any of its braided module categories with a canonical (though not unique) balancing.

Examples of braided module categories include the following:
\begin{enumerate}
\item The category $\cA$ itself is a braided $\cA$-module.  It corresponds to the ``vacuum marking", and is an essential component in our computation for unmarked surfaces.  
\item For any surface $S^\circ$ with circle boundary, the category $\int_{S^\circ}\cA$ is a braided module category, by insertions of annuli along the boundary.  In \cite{Ben-Zvi2015}, we identified $\int_{S^\circ}\cA$ with the category of modules for an explicitly constructed algebra $A_{S^\circ}$.
\item Quantizations of conjugacy classes in $G$, following \cite{Donin2002a,Donin2006,Donin2004, Kolb2009}, define braided module categories.  An important example is the so-called Ruijsenaars-Schneider conjugacy class, consisting of matrices which differ from the identity by a matrix of rank at most one \cite{Varagnolo2010,Jordan2014, Balagovic2016}.
\item Examples of braided module categories related to a variant of the trigonometric Knizhnik--Zamolodchikov equation, and the theory of dynamical quantum groups, appear in \cite{Brochier2012,Brochier2013,Enriquez2008}.
\item A ``boundary condition" is the local marking data for a half-plane.  Algebraically, this is the data of a tensor category $\cB$ attached to the boundary, together with a braided tensor functor $\cA\to Z(\cB)$ to the Drinfeld center of $\cB$.  The trace, or Hochschild homology category, of such a $\cB$ carries the structure of a braided module category for $\cA$. Important examples are provided by parabolic subgroups.
\end{enumerate}

In \cite{Ben-Zvi2015}, we identified $\int_{Ann}\cA$ with the category of modules for the ``reflection equation algebra" $\oa \in \cA$.  In the case $\cA=\Repq  G$, $\oa$ is a quantization of the coordinate algebra $\cO(G)$, equipped with its Semenov-Tian-Shansky Poisson bracket. In Section  \ref{annulus-category-section}, we prove the following theorem, giving yet a third reformulation of the notion of a braided module category, in terms of $\oa$:

\begin{theorem}\label{quantum-moment-map-theorem}
Let $\cM$ be a braided module category.
\begin{enumerate}
\item For every $M\in \cM$, we have a canonical homomorphism of algebras,
$$\mu_M:\oa\to \underline{\End}_\cA(M).$$
We call $\mu_M$ \emph{the quantum moment map} attached to $M$.
\item Assuming $M$ is a progenerator for the $\cA$-action\footnote{equivalently, for the $\int_{Ann}\cA$-action; see Theorem~\ref{inducedactions}}, we moreover have an equivalence of $\int_{Ann}\cA$-module categories,
$$\cM\simeq \underline{\End}_\cA(M)\modu_{\int_{Ann}\cA}.$$

\item The action of any $X\in \int_{Ann}\cA$ on any $N\in \cM \simeq\underline{\End}_\cA(M)\modu_{\int_{Ann}\cA}$ is given by relative tensor product,
$$N\bt X\mapsto N\underset{\oa}{\ot}X,$$
over the homomorphism $\mu_M$.
\end{enumerate}
Conversely, given an algebra $A\in\cA$ and a homomorphism $\mu:\oa\to A$, the category $\cM = A\modu_\cA$ is equipped with the structure of a braided module category, with action as in (3).  The regular $A$-module $A\in\cM$ is an $\cA$-progenerator in this case.
\end{theorem}

Natural examples of quantum moment maps arise in the following contexts:

\begin{enumerate}
\item The quantum moment map for $\cA$ itself is the co-unit homomorphism, $\epsilon:\oa\to \mathbf{1}_\cA,$ which quantizes the homomorphism of evaluation of a function at the identity element.
\item In Section~\ref{reconstruction-section} , we obtain canonical quantum moment maps $\mu: \oa \to A_{S^\circ}$ which control the braided module category structure on $\int_{S^\circ}\cA$.  These quantize the classical multiplicative moment maps, which send a local system to its monodromy around the puncture.
\item By their construction -- as equivariant quotients of the reflection equation algebra -- quantizations of conjugacy classes carry canonical quantum moment maps.
\end{enumerate}

\subsection{Computing factorization homology of marked surfaces}
Let us fix a surface $S$ equipped with a marked point $x\in S$, a disc $D_x$ containing $x$, and a braided module category $\cM$.  Let $S^\circ=S\setminus x$, and fix a disc embedding $i_x:D_x\subset S$, and resulting annulus embedding $Ann\simeq D_x\setminus x\subset S$.  We may then compute the factorization homology using excision,

$$\int_{(S,x)}\!\!\!\!\!\!\!(\cA,\cM_x) \simeq \int_{S^\circ}\!\!\!\cA \underset{\scalebox{0.75}{$\displaystyle{\int_{Ann}\!\!\!\!\!\!\cA}$}}{\bt} \cM_x.$$

Building on Theorem~\ref{quantum-moment-map-theorem} we can describe the tensor product above explicitly, in the language of quantum Hamiltonian reduction.  

\begin{theorem} Let $i_{x*}:\cM\to\int_{(S,x)}(\cA,\cM)$ denote the push-forward in factorization homology along the embedding $i_x$.
\begin{enumerate}
\item For any $M\in\cM$, we have a natural isomorphism:
$$\End\left(i_{x*}(M)\right) \cong A_{S^\circ}\dS_{\!\!\!\mu}M:=\Hom(\mathbf{1}_\cA, A_{S^\circ}\underset{\oa}{\ot} \underline{\End}_\cA(M)),$$
of the endomorphism algebra of $i_{x*}(M)$, in the category $\int_{(S,x)}(\cA,\cM_x)$ with the quantum Hamiltonian reduction of $A_{S^\circ}$ along the quantum moment map $\mu_M$.

\item Suppose that $M$ is an $\cA$-progenerator.  Then we have equivalences of categories,
\begin{align*}\int_{(S,x)}\!\!\!\!\!\!\!\!\!\left(\cA,\cM\right) \quad\simeq\quad \int_{S^\circ}\!\!\!\cA \underset{\scalebox{0.75}{$\displaystyle{\int_{Ann}\!\!\!\!\!\cA}$}}{\bt} \cM\quad\simeq\quad (A_{S^\circ}\bimodu \underline{\End}(M))_{\oa\modu},\end{align*}
with the category of bimodules for $A_{S^\circ}$ and $\underline{\End}(M)$, in the category $\oa\modu$.  

\item The quantum global sections functor
$$\Gamma = \Hom(i_*(M),-): \int_S\cA \to (A_{S^\circ}\dS_{\!\!\!\!\mu} M)^{op}\modu,$$
valued in the category of modules for the quantum Hamiltonian reduction is naturally equivariant for actions of the marked-and-colored mapping class group of the surface.
\end{enumerate}
\end{theorem}

This Theorem (more generally in the case of several marked points) is proved in Section~\ref{reconstruction-section}.

In particular we obtain a description of the category associated to a closed unmarked surface $S$. Choose some disk $D^2\subset S$, and let $S^\circ$ denote its complement in $S$. 
\begin{corollary}
We have an equivalence of categories,
$$\int_S\cA \simeq \int_{S^\circ}\!\!\!\cA \underset{\scalebox{0.75}{$\displaystyle{\int_{Ann}\!\!\!\!\!\cA}$}}{\bt} \cA \simeq (A_{S\setminus D^2}\bimodu \mathbf{1}_\cA)_{\oa\modu},$$
with the category of bimodules for $A_{S^\circ}$ and $A_{D^2}=\mathbf{1}_\cA$, in the category $\oa\modu$.  
\end{corollary}

Likewise we can identify global functions on the quantum character variety (i.e., endomorphisms of the quantum structure sheaf $\cO_{\cA,S})$ with the quantum Hamiltonian reduction of $A_{S^\circ}$ along the quantum moment map $\mu$, $$\End(\cO_{\cA,S})\cong \left(A_{S^\circ} \Big{/} A_{S^\circ}\cdot\mu(\ker(\epsilon))\right)^{inv}.$$
Thus we have a global sections functor 
$$\Gamma = \Hom(\cO_{\cA,S},-): \int_S\cA \to (\cA_{S^\circ}\dS_{\!\!\!\mu}\mathbf{1}_\cA)^{op}\modu,$$
valued in the category of modules for the quantum Hamiltonian reduction, equivariant for an action of the mapping class group of the surface.

\subsection{Quantization of character varieties}
In the classical setting $\cA=\Rep G$, it was a fundamental observation of~\cite{Alekseev1998} that $G$ character varieties of closed surfaces could be obtained via ``multiplicative Hamiltonian reduction" of their punctured counterparts.  Let us briefly recall the classical construction here.

Let $S^\circ$ be a surface with one distinguished circle boundary component with a point $p$ chosen on it. Let $\RV(S^\circ)$ denote the representation variety of $S^\circ$, i.e.
\[
	\RV(S^\circ)=\{\rho:\pi_1(S,p)\rightarrow G \}.
\]
Equivalently, $\RV(S^\circ)$ is the variety of $G$-local systems on $S^\circ$, equipped with a trivialization of the fiber at $p$.  Changing the choice of trivialization amounts to conjugating a given homomorphism by a group element.  Hence, the $G$-character stack of $S$ is the quotient stack
\[
	\uch_G(S^\circ)=\RV(S^\circ)/G.
\]
The embedding of the annulus around the circle boundary of $S$ induces a $G$-equivariant map,
\[
	\mu:\RV(S^\circ)\longrightarrow \RV(Ann)= G.
\]
The map $\mu$ is called a ``multiplicative" or ``group-valued" moment map, in~\cite{Alekseev1998}.

Fix a conjugation invariant subvariety $C\subset G$ (i.e., a union of conjugacy classes).  Then $\mu^{-1}(C) \subset \RV(S)$ is a $G$-stable subvariety. The character stack of $(S,C)$ is then the quotient stack
\[
	\uch_G(S,C)=	\mu^{-1}(C)/G.
\]
In other words, $\uch_G(S,C)$ is a moduli stack of $G$-local systems on $S^\circ$ whose monodromy around the boundary lies in $C$.  By definition, the category $\QCoh(\uch_G(S,C))$ is the category of $G$-equivariant quasi-coherent sheaves on $\mu^{-1}(C)$.  The variety $\mu^{-1}(C)/G$ is called the ``multiplicative Hamiltonian reduction" of $\RV(S^\circ)$ along $\mu$.

The main results of \cite{Ben-Zvi2013, Ben-Zvi2015}, give identifications,
$$\int_{S^\circ} G\modu \simeq \QCoh(\uch(S^\circ)),\qquad A_{S^\circ} \cong \cO(G^{2g}).$$

The category $\QCoh(C/G)$ is a braided module category for $\Rep G$, which we may associate to the puncture, and the category $\QCoh(\mu^{-1}(C)/G)$ is precisely the category produced by factorization homology for a marked surface $(S,x)$, where the marked point is decorated with $\QCoh(C/G)$.

When we instead take $\cA=\Repq G$, excision gives rise to ``quantum multiplicative Hamiltonian reduction", as discussed in e.g.~\cite{Jordan2014}. Therefore, we obtain quantizations of character stacks of closed surfaces.  Taking global sections passes to the affinization of the character stack, the Poisson variety obtained as the categorical quotient $\RV(S)\sslash G.$  More generally if $C$ is any conjugacy class, the categorical quotient $\mu^{-1}(C)\sslash G$ is a symplectic leaf of the Poisson variety $\RV (S)\sslash G$ (the case of the closed surface corresponds to $C=\{e\}$. In~\cite{Donin2006,Donin2004,Donin2002a} an explicit quantization of any given conjugacy class is given, using Verma modules: by construction these come equipped with an equivariant algebra map from $\cO_q(G)$. By Theorem~\ref{quantum-moment-map-theorem} the category of equivariant modules over this algebra is a braided module category, hence its factorization homology over the marked surface as above gives a quantization of the variety $\mu^{-1}(C)\sslash G$.

\subsection{Towards quantum character sheaves}
The invariant assigned to the torus $S=T^2$ plays a central role in topological field theory. 
In three-dimensional Chern-Simons/Witten-Reshetikhin-Turaev theory at level $k$ the invariant of $T^2$ is identified on the one hand with the Verlinde algebra, the group of characters of integrable level $k$ representations of the loop group (which themselves form the invariant of $S^1$, i.e., the Wilson lines or codimension two defects). The natural mapping class group symmetry of the torus invariant then explains the well-known modularity of these characters. On the other hand the Freed-Hopkins Teleman theorem~\cite{Freed2011} identifies the $T^2$ invariant with a version of class functions on the compact group, namely the twisted equivariant $K$-theory of $G_c/G_c$. In this section we describe the corresponding roles of the quantum character variety of the torus.

Let us consider the oriented field theory defined by integrating a balanced tensor category $\cA$ on oriented surfaces.  
We have identified codimension two defects for quantum $\cA$-character varieties with braided $\cA$-modules, i.e., modules for 
$U(\cA)=\int_{Ann}\cA$, which thanks to the balancing is also identified (as monoidal category) with the cylinder integral $\int_{Cyl}\cA$ (see Remark~\ref{different framings}). In the language of extended topological field theory, this means we attach the 2-category $U(\cA)\modu$ to the circle.

The excision axiom applied to a decomposition of $T^2$ into cylinders allows us to identify the torus integral as the monoidal Hochschild homology, or trace, of the cylinder (hence annulus) integral of $\cA$:
$$\int_{T^2}\cA\simeq U(\cA)\boxtimes_{U(\cA)\boxtimes U(\cA)^{op}} U(\cA)= Tr(U(\cA)).$$ Equivalently, we can describe the torus integral as the trace (or Hochschild homology) of the 2-category of braided $\cA$-modules. It follows from the general theory of characters in 
Hochschild homology (see e.g.~\cite{Ben-Zvi2013} for references) that $\int_{T^2}\cA$ carries characters 
$[\cM]\in \int_{T^2}\cA$ for sufficiently finite (i.e., dualizable) braided $\cA$-modules $\cM$. 

\begin{remark}[Braided $G$-categories and loop group categories] The 2-category of braided $\Repq  G$-modules 
is a 4d gauge theory analog of the category of integrable level $k$ representations of the loop group. Indeed, it is the Betti form~\cite{Ben-Zvi2016a} of the 2-category of chiral module categories over the Kazhdan-Lusztig category of integrable representations of the loop algebra, which itself is a form of the local geometric Langlands 2-category of categories with an action of the loop group at level $k$ (where $q$ is essentially $\exp(2\pi i/k)$), see~\cite{Gaitsgory2016}.
\end{remark}

\medskip

In the case $\cA=\Repq  G$, in \cite{Ben-Zvi2015} we identified the punctured torus category with the category of modules in $\Repq  G$ for the algebra of quantum differential operators, considered as an algebra object in $\Repq  G$ under the adjoint action:
$$\int_{T^2\setminus D^2}\Repq (G)\simeq \cD_q(G)\modu_{\Repq  G}.$$  
Note that since we are considering modules in $\Repq  G$ rather than $Vect$, this category is a quantum analog not of the category of $\cD$-modules on $G$ but of the category of $\cD$-modules on $G$ which are weakly equivariant for the adjoint action (from which the former can be obtained by de-equivariantization). It follows from the quantum Hamiltonian reduction formalism of the previous section that sealing up the puncture results in imposing the quantum moment map relations for the adjoint action -- i.e., in imposing {\em strong} equivariance.

  We define:

\begin{definition}
The category $\cD_q(\frac{G}{G})\modu$ of strongly ad-equivariant $\cD_q(G)$-modules has its objects pairs $(M,\phi)$ consisting of a $\cD_q(G)$-module $M\in\cD_q(G)\modu_\cA$, and an action map, $\phi:M\underset{O_\cA}{\ot} \mathbf{1}_\cA\to M$ in the category $\cD_q(G)\modu_\cA$, satisfying the associativity conditions making $M$ into an $\mathbf{1}_{\cA}$-module.
\end{definition}

As a corollary of Theorem~\ref{quantum-moment-map-theorem}, we have:

\begin{theorem}\label{ad-equiv}\label{torus-int}
We have equivalences of categories,
$$\int_{T^2}\Repq  G \simeq Tr(U(\cA))\simeq \cD_q(\frac{G}{G})\modu.$$
\end{theorem} 

In particular $\cD_q(\frac{G}{G})\modu$ inherits an action of $\widetilde{SL_2(\ZZ)}$, including a quantum Fourier transform ($S$-transformation) generalizing the difference Fourier transform in the case $G=H$ a torus. Indeed the endomorphisms of the quantum structure sheaf are known in many cases (see below for the $t$-analog) to recover $$\End(\cO_{\Repq  G,T^2}) \cong \cD_q(H)^W,$$
the algebra of $W$-invariant $q$-difference operators on the torus $H$.  This should be compared to the computation of~\cite{Frohman2000}.

It follows that, in analogy with the Freed-Hopkins-Teleman theorem for Chern-Simons theory, 
the characters of braided $\Repq  G$-modules form quantum $\cD$-modules on $G/G$.
This is a quantum analog of the interpretation~\cite{Ben-Zvi2009} of Lusztig character sheaves in $\cD(G/G)\modu$ in terms of module categories for $\cD$-modules on $G$. We likewise expect a theory of quantum character sheaves to provide a natural $q$-analog of the Lusztig theory. Interesting examples of such quantum character sheaves are provided by the quantum Springer sheaves --- the characters of the braided module category $\tr(\Repq  B)$ associated to the $\Repq  G$-algebra defined by the quantum Borel (or other parabolics), which can be expected via a quantum Hotta-Kashiwara theorem to be described by a q-analog of the Harish-Chandra system.  In particular one expects the entire torus category $\cD_q(\frac{G}{G})\modu$ to carry a ``quantum generalized Springer" orthogonal decomposition into blocks labeled by cuspidal objects associated to Levi subgroups, in analogy with the results of~\cite{Gunningham2015} for $\cD({\mathfrak g}/G)$.

\subsection{The double affine Hecke algebra}
The double affine Hecke algebra (abbreviated DAHA, and denoted $\mathbb{H}_{q,t}$) associated to $G=GL_N$ (or more generally to a reductive group $G$) is a celebrated two-parameter deformation of the group algebra of the double affine Weyl group of $G$, introduced by Cherednik. It contains as a subalgebra the spherical DAHA (denoted $\mathbb{SH}_{q,t}$), which is a flat one-parameter deformation of the algebra $\cD_q(H)^W$ of $W$-invariant $q$-difference operators on the torus $H\subset G$. 
The spherical DAHA for $GL_n$ appears naturally~\cite{Oblomkov2004} as a quantization of the
phase space of the trigonometric Ruijsenaars-Schneider system~\cite{Ruijsenaars1986,Fehér2013,Feher2012,Feher2014} (also known as the relativistic version of the trigonometric Calogero-Moser system), a many-body particle system with multiplicative dependence on both positions and momenta. The phase space of this integrable system in turn has a well-known interpretation~\cite{Gorsky1995,Fock1999} in terms of the character variety of the torus, marked by a distinguished ``mirabolic" conjugacy class at one point. Namely it is identified with a space of ``almost-commuting" matrices, invertible matrices whose commutator lies in a minimal conjugacy class (differing from a scalar matrix by a matrix of rank one).

In this section, we will explain how our theory of quantum character varieties naturally 
produces the spherical DAHA when fed the torus $T^2$ marked by the quantum 
mirabolic conjugacy class. 

It is known that the spherical DAHA $\mathbb{SH}_{q,t}$ may be obtained from the algebra $\cD_q(G)$ of quantum differential operators on $G$ by quantum Hamiltonian reduction along the quantum moment map $\mu_q:\cO_q(G)\rightarrow \cD_q(G)$ at a certain equivariant two-sided ideal $\mathcal{I}_t\subset \cO_q(G)$, depending on a parameter $t$.  The ideal $\mathcal{I}_t$ is a canonical $q$-deformation of the variety of matrices which differ from the scalar matrix $t\cdot \id$ by a matrix of rank at most one.  These results give rise to multiplicative analogues of the relation between the trigonometric Cherednik algebra (quantizing the trigonometric Calogero-Moser phase space) and mirabolic $\cD$-modules, see~\cite{Gan2006,Etingof2007,Finkelberg2010}.

\begin{theorem}\label{known-results} We have an isomorphism of algebras $$(A_{(T^2)^\circ}/\mathcal{I}_{t})^{U_q\mathfrak{gl}_N} \cong \mathbb{SH}_{q,t}(GL_N),$$
in the following settings:
\begin{enumerate}
\item When $q$ is a root of unity \cite{Varagnolo2010},
\item When $q=e^\hbar$ ($\hbar$ formal) for the Drinfeld-Jimbo category \cite{Jordan2014},
\item For arbitrary $q\in\CC^\times$ when $N=2$ \cite{Balagovic2016}.
\end{enumerate}
\end{theorem}

More precisely, we have slightly reformulated each result here, for a more uniform presentation.  Let us spell out the dictionary here, for the reader's convenience.  In \cite{Varagnolo2010}, an algebra of quantum differential operators on $GL_N\times \mathbb{P}^{N-1}$ is constructed, along with a quantum moment map.  In the notation of \cite{Jordan2014}, this same algebra is denoted $\cD_q(\Mat_{\mathbf{d}}(Q))$, for $(Q,\mathbf{d}) = \overset{1}{\bullet}\rightarrow\overset{N}{\bullet}\rightloop$, and is an important special case of the quiver construction.  We have isomorphisms of algebras,
$$\cO_q(G)\cong \cO_q(\overset{N}{\bullet}\rightloop) \cong A_{Ann},\qquad\cD_q(GL_N) \cong \cD_q(\overset{N}{\bullet}\rightloop) \cong A_{(T^2)^\circ},$$
where each the first isomorphisms is clear from inspection of the defining relation, and while the second are special cases of the main result of \cite{Ben-Zvi2015}.  The $\mathbb{P}^{N-1}$ factor in~\cite{Varagnolo2010}, and the extra vertex on the quiver in \cite{Jordan2014}, each give rise to the deformed ideal $\mathcal{I}_t$, and so in each of the three cases of the theorem, one obtains the same quantum Hamiltonian reduction algebra.

By Theorem~\ref{quantum-moment-map-theorem}, we may define a braided module category $\cM_t=A_t\modu_\cA$, where $A_t:=\cO_q(G)/\mathcal{I}_t$ comes equipped with a canonical quantum moment map, as a quotient of $\cO_q(G)$. Finally, the quantum Hamiltonian reduction computed in those papers is precisely that which appears in Theorem~\ref{known-results}.  We therefore obtain the following important corollary:

\begin{corollary}
Let $\cA=\Repq  GL_n$, $\cM=\cM_t$, $S=(T^2,x)$ the closed torus, with a single marked point $x$ colored by $\cM_t$.  Then we have an isomorphism,
$$\End(i_{x*}(A_t)) \cong \mathbb{SH}_{q,t}(GL_n),$$
of the endomorphism algebra of $i_{x*}(A_t)$ as an object of $\int_{(T^2,x)}(\cA,\cM_t)$, and spherical DAHA $\mathbb{SH}_{q,t}(GL_n)$.

Hence, we obtain a marked mapping class group-equivariant ``global sections" functor
$$\Gamma:  \int_{(T^2,x)}(\cA,\cM_t)\to \mathbb{SH}_{q,t}(GL_n)\modu,$$
$$\qquad M\quad \mapsto\quad \Hom(i_{x*}(A_t), M)$$
from factorization homology of the marked torus to spherical DAHA-modules.
\end{corollary}

More generally, the category $\int_{(T^2,x)}(\cA,\cM_t)$ provides a quantum version of the category of mirabolic $\cD$-modules studied in~\cite{Nevins2009,Finkelberg2010,Bellamy2015} and others, of which representations of spherical DAHA provide the ``principal series" part.

This result gives a topological explanation of the existence of a quantum Fourier transform on $\mathbb{SH}_{q,t}(GL_n)$ leading to an action of the marked torus mapping class group $\widetilde{SL_2(\ZZ)}$ by algebra automorphism.  It also justifies the moniker ``operator-valued Verlinde algebra", by which Cherednik first referred to his DAHA~\cite{Cherednik2005,Cherednik2004}: while the Verlinde algebra is attached to $T^2$ by the 3d Witten-Reshetikhin-Turaev theory, we see that the spherical double affine Hecke algebra is the affinization of the category attached to a marked $T^2$ by the 4D theory, so that it obtains all the topological symmetries of the torus from functoriality of the construction.


\subsection{Acknowledgments}
We would like to thank Pavel Etingof and Benjamin Enriquez for sharing their ideas concerning elliptic structures on categories.  We'd also like to thank John Francis, Greg Ginot, Owen Gwilliam, and Claudia Scheimbauer for discussions about factorization homology.

The work of DBZ and DJ was partly supported by NSF grant DMS-1103525 and ERC Starting Grant 637618. DBZ and DJ would like to acknowledge that part of the work was carried out at MSRI as part of the program on Geometric Representation Theory.

\section{Factorization homology of surfaces}
In this section we briefly review factorization homology of stratified $n$-manifolds, following \cite{Lurie,Francis2012, Ayala2012}.  We will put special emphasis on the case $n=2$, of surfaces -- possibly marked and/or with boundary -- and with values in certain 2-categories $\Rex$ or $\LFP$, of $\K$-linear categories (see \cite{Ben-Zvi2015} for a review of $\Rex/\LFP$ as settings for factorization homology):  in this case, many of the structures demanded by the general framework of factorization homology recover well-known structures in quantum algebra.

As in \emph{loc. cit.} our main example will be the balanced tensor category $\Repq G$: this notation means we choose a reductive algebraic group $G$, a Killing form $\kappa$ on $\mf g = Lie(G)$, and consider either the category of finite-dimensional $U_q(\mf g)$-modules,  when $G$ is simply connected, or the corresponding braided tensor subcategory when $G$ is not semisimple.  We do not recall a presentation of $U_q(\mf g)$ here, but rather refer to e.g.~\cite[Section~9.1]{Chari1994} for basic definitions.

\subsection{Factorization homology of surfaces}\label{sec:factHom}

Let $s$ stand in for either framing (fr) or orientation (or). We denote by $\Mfld^{2}_{s}$ the $(\infty,1)$-category, whose objects are (framed or oriented) 2-dimensional manifolds with corners and whose morphisms spaces are the $\infty$-groupoids of (framed or oriented) embeddings.  We denote by $\Disk^2_{s}$ the full subcategory whose objects are arbitrary (possibly empty) finite disjoint unions of $\mathbb{R}^2$.  Each category is naturally symmetric monoidal with respect to disjoint union.

\begin{definition}
	A $\operatorname{Disk}_s^2$-algebra in an $(\infty,1)$ symmetric monoidal category $\cC$, for $s\in \{fr,or\}$, is a symmetric monoidal functor from $\operatorname{Disk}^2_{s}$ to $\cC$. 
\end{definition}

\begin{remark}\label{rmk:MeqA} A $\operatorname{Disk}^2_{fr}$-algebra (or rather the image of $\RR^2$) is usually called an $E_2$-algebra, or algebra over the little disk operad. Similarly a $\operatorname{Disk}^2_{or}$-algebra is an algebra over the framed little disk operad.
\end{remark}

The data of $\cA$ is completely determined, in the framed case, by the image $\cA(\mathbb{R}^2)$ of the generator $\mathbb{R}^2$, and a collection of morphisms $\cA^{\bt k}\to\cA$ (including $k=0$, which gives the unit map), and a well-known host of coherences.  We abuse notation, and denote both the functor and its value on the generator $\mathbb{R}^2$ by the symbol $\cA$.  In the oriented case, we have also to specify the ``balancing" automorphism of the identity functor. This corresponds to the loop, in the space of oriented diffeomorphisms of a disc, which rotates $\theta$ degrees about the origin, for $\theta\in [0,2\pi]$.
\begin{remark}
	Since our target is a (2,1)-category, functors from $\Mfld^2_s$ factors through its ``homotopy (2,1)-category'', i.e. the category whose Hom spaces are fundamental groupoids of spaces of embeddings.
\end{remark}
Our main case of interest is when $\cC=\Rex$ or $\LFP$ is a certain symmetric monoidal $2$-category of $\K$-linear categories with the Kelly-Deligne tensor product (see Section \ref{reconstruction-section} below, and Section 3 of \cite{Ben-Zvi2015} for details).  In this case, the data of an $E_2$-algebra consists of a braided tensor category structure on $\cA(\mathbb{R}^2)$, which we denote simply as $\cA$, by abuse of notation.  The data of a $\Disk^2_{or}$-algebra is identified with a {\em balanced} braided tensor structure.  

The factorization homology $\int_S \cA$ of an $E_2$-algebra (resp, $\Disk^2_{or}$-algebra) $\cA$ on a framed (resp, oriented) surface $S$ is defined as a colimit,
$$\int_S \cA =\colim{(\mathbb{R}^2)^{\sqcup k}\to S} \cA^{\bt k},$$
over all framed (resp. oriented) embeddings of disjoint unions of discs into $S$, where the (1-, and 2-) morphisms in the diagram are comprised of the tensor functors $\cA^{\bt^k}\to\cA$, and their coherences (including the associativity and braiding isomorphisms).

In other words (in the case $\cC=\Rex/\LFP$), it is the universal recipient of functors from $\cA^{\bt k}$ labeled by collections of disjoint disks in $S$, and factoring through the $\Disk^2_{or}$-algebra structure on $\cA$, whenever a disk embedding factors through inclusions of disks in a larger disk.
Formally this colimit is expressed as a  left Kan extension,

$$\xymatrix{
\Disk^2_{s} \ar@{^{(}->}[dr] \ar[rr]^{\cA} && \cC\\
& \Mfld^2_{s}\ar@{-->}[ur]_{\scalebox{0.75}{$\displaystyle{\int_{-}\cA}$}}
}.$$

An important feature of factorization homology is that the empty set is regarded as a surface, and has an initial embedding to any surface.  This induces a map $\mathbf{1}_\cC\simeq \int_\emptyset \cA\to \int_S\cA$, for any $S$.  In the case $\cC=\Rex/\LFP$, we have $\mathbf{1}_\cC=\Vect$, so the initial functor is determined by the the image of $\K\in\Vect$.  This equips all braided tensor categories appearing in the theory with their unit object, and equips factorization homology of any surface $S$ with a distinguished object, which we called in \cite{Ben-Zvi2015} the ``quantum structure sheaf," $$\cO_{S,\cA}\in \int_S\cA.$$  It follows that we can also calculate $\cO_{\cA,S}$ as the image of the unit in $\cA$ under the map $\cA\to\int_S\cA$ associated to any disc embedding.

Factorization homology satisfies an important excision property.
Given a 1-manifold $P$, the factorization homology $\int_{P\times I} \cA$ on the cylinder (with any framing) carries a canonical $E_1$ (associative) algebra structure from the inclusion of disjoint unions of intervals inside a larger interval (i.e., we stack cylinders inside a larger cylinder, see Figure \ref{stacking-picture}). Moreover the invariant of a manifold with a collared boundary $M$ is naturally a module over $\int_{P\times I} \cA$. 

This structure allows us to describe the factorization homology of a glued (framed or oriented) surface $S=S_1\sqcup_{P\times I} S_2$ as a relative tensor product:

$$\int_{S}\!\!\cA = \int_{S_1}\!\!\cA \underset{\scalebox{0.75}{$\displaystyle{\int_{P\times I}\!\!\!\!\cA}$}}{\bt} \int_{S_2}\!\!\cA.$$

In the case of punctured surfaces studied in \cite{Ben-Zvi2015}, we exploited excision for $P=I$, an interval, to compute factorization homology categories as categories of modules for certain explicitly presentable ``moduli algebras".  In the present paper, we will extend these descriptions to closed surfaces, and this involves applying excision in the case $P=S^1$.  For that case, we need to develop the theory of surfaces with marked points, and an explicit description for the tensor structure on the annulus.  

\subsection{Marked points}
Following \cite{Ayala2012}, factorization homology for surfaces with marked points may be defined similarly as for unmarked surfaces, via a Kan extension from a category of marked discs to a category of marked surfaces.

We denote by $\Disk^2_{s,mkd}$ the $(\infty,1)$-category whose objects are disjoint unions of unmarked (framed or oriented) disks $\mathbb{R}^2$  and once-marked disks $\mathbb{R}^2_{\mathbf{0}}$, and whose morphism spaces are spaces of (framed or oriented) embeddings, which are moreover required to send marked points bijectively to marked points. 

The data of a symmetric monoidal functor $\Disk^2_{fr,mkd}\to \cC$ is equivalent to an $E_2$-algebra $\cA$ (the restriction to unmarked disc) and an object $\cM\in \cC$ assigned to a marked disc, equipped with the structure of {\em $E_2$-module} over $\cA$ (the intrinsic operadic notion of module for an $E_2$-algebra). That is, $\cM$ is equipped with a compatible collection of functors $\cM\bt\cA^{\bt k}\to\cM$  determined by choosing embeddings from the disjoint union of one marked and $k$ unmarked, disks back into the marked disk, satisfying a collection of coherences.  By analogy, we will call the image of the once marked disk through a functor $\Disk^2_{or,mkd}\to\cC$ a $\Disk^2_{or}$-module over the $\Disk^2_{or}$-algebra image of the unmarked disk.

\begin{remark} The requirement that marked points map bijectively means that the empty set is no longer initial in the category of marked surfaces.  Allowing maps which are only injective on marked points is equivalent to giving a pointing $1_\cC\to \cM$ with no additional coherences, and agrees with the notion of locally constant factorization algebra on the stratified space $\mathbb{R}^2_{\mathbf{0}}$.
\end{remark}

Similarly to ordinary factorization homology, a symmetric monoidal functor from $\Disk^2_{s,mkd}$ to some target category $\cC$ is determined by its values $\cA$, and $\cM$ on the unmarked, and once-marked discs, respectively, together with a host of functors and coherences between various tensor products of $\cA$ and $\cM$.  Let us therefore denote such a functor by the pair $(\cA,\cM)$.

\begin{definition}
The factorization homology of the pair $(\cA,\cM)$ is the left Kan extension, 
$$\xymatrix{
\Disk^2_{or,mkd} \ar@{^{(}->}[dr] \ar[rr]^{(\cA,\cM)} && \cC\\
& \Mfld^2_{or,mkd}\ar@{-->}[ur]_{\scalebox{0.75}{$\displaystyle{\int_{-}(\cA,\cM)}$}}
}.$$
\end{definition}

As for ordinary factorization homology of surfaces, the definition by left Kan extension implies a formula for the factorization homology of any marked surface as a colimit,
$$\int_{(S,X)}\!\!\!(\cA,\cM) = \colim{(\mathbb{R}^2)^{\sqcup k}\sqcup(\mathbb{R}^2_\mathbf{0})^{\sqcup l}\to (S,X)} \cA^{\bt k}\bt \cM^{\bt l},$$
over all embeddings of unmarked or once-marked discs into $S$.

Just as for ordinary factorization homology, we have an excision property for computing factorization homology of marked surfaces.  Let $$(S,X)= (S_1,X_1)\cup_{P\times I} (S_2,X_2)$$ be the relative union of marked surfaces $(S_1,X_1)$ and $(S_2,X_2)$, along some (unmarked) cylinder $P\times I$.  Then we have

$$\int_{(S,X)}\!\!\!\!\!\!\!(\cA,\cM) = \int_{M_1}\!\!\!(\cA,\cM) \underset{\scalebox{0.75}{$\displaystyle{\int_{P\times I}\!\!\!\!\!\!\cA}$}}{\bt} \int_{M_2}\!\!\!(\cA,\cM).$$

\subsection{Boundary conditions}\label{boundaries}
An important source of examples of markings of surfaces come from boundary conditions in the topological field theory defined by $\cA$, or concretely from factorization algebras on manifolds with boundary, extending $\cA$ from the interior. 

We denote by $\Disk^2_{fr/otd,bdry}$ the  $(\infty,1)$-category whose objects are disjoint unions of unmarked (framed or oriented) disks $\mathbb{R}^2$  and half-spaces $\mathbb{H}=\mathbb{R}\times\mathbb{R}_{\geq 0}$, and whose morphism spaces are spaces of (framed or oriented) embeddings, which are moreover required to respect boundaries. 
The data of a symmetric monoidal functor $\Disk^2_{fr,bdry}\to \cC$, i.e., of a factorization algebra on the stratified space $\mathbb{H}$, is equivalent to an $E_2$-algebra $\cA$ (the restriction to unmarked disc) and an object $\cB\in \cC$ assigned to a half-space, equipped with the structure of {\em $\cA$-algebra}: an algebra object in $\cC$ equipped with an $E_2$-morphism $z:\cA\to \cZ(\cB)$ from $\cA$ to the center $$\cZ(\cB)=End_{\cB\bt \cB^{op}}(\cB)$$of $\cB$ (i.e. the pair $(\cA,\cB)$ is an algebra over Voronov's Swiss--cheese operad~\cite{Voronov1999}).  Here the algebra structure on $\cB$ comes from the inclusion of unions of half spaces into half-spaces -- i.e., $\cB$ itself defines a one-dimensional factorization algebra valued in $\cC$. The central action comes from the inclusion of a disc into the half space.

In the case $\cC=\Rex$, this means we have a braided tensor category $\cA$, a tensor category $\cB$, and a functor of braided tensor categories from $\cA$ to the Drinfeld center of $\cB$.

Two rich sources of $\cA$-algebras, hence boundary conditions, for $\cA$ a braided tensor category are:
\begin{enumerate}
\item Categories of modules $\cB=B\modu$ for commutative (i.e., braided or $E_2$) algebra objects $B\in \cA$. 
\item The category $\cB=\Repq  B$, of torus-integrable representations of the quantum Borel subalgebra $U_q(\mathfrak{b}^+)\subset U_q(\mathfrak{g})$ form a $\Repq  G$-algebra.  This follows simply from the fact that
$$\mathcal{R} \in U_q(\mathfrak{b}^+) \ot U_q(\mathfrak{b}^-) \subset U_q(\mathfrak{g}) \ot U_q(\mathfrak{g}),$$
where $\mathcal{R}$ is the quantum $R$-matrix which controls the braiding on $\Repq G$.  Hence $\mathcal{R}$ itself endows the forgetful functor $\Repq  G\to \Repq B$ with the structure of a central functor.  More generally quantum parabolic subgroups define boundary conditions.
\end{enumerate}

Given a surface with boundary $(S,\partial S)$ we can perform factorization homology $$\int_{(S,\partial S)} (\cA,\cB)$$ for the pair $(\cA,\cB)$ following the same formalism as in the unmarked and marked cases. 

An $\cA$-algebra $\cB$ is itself a one-dimensional factorization algebra and so can be integrated on closed one-manifolds
$P\mapsto \int_P\cB$. In particular we have the trace / cocenter / Hochschild homology category $$\tr(\cB)= \int_{S^1}\cB\simeq\int_{(S^1\times {\mathbb R}_{\geq 0},S^1)} (\cA,\cB).$$
We thus see that $\tr(\cB)$ is naturally an $\int_{S^1\times \RR}\cA$-module. When $\cA$ is balanced, this is identified with the annulus category $U(\cA)$, and thus we find
\begin{proposition}
Let $\cA$ denote a ribbon category and $\cB$ an $\cA$-algebra.
\begin{enumerate}
	\item The trace $\tr(\cB)$ carries a natural structure of $\Disk^2_{or}$ $\cA$-module. 

	\item Let $(S,\partial S)$ denote a compact surface with boundary $\partial S$ identified with $\coprod_n (S^1)$, and $(\overline{S},\{x_i\})$ the closed surface with marked points obtaining by sewing in discs along $\partial S$. Then the factorization homology of $\cA$ on $S$ with boundary marked by $\cB$ agrees with that of $\cA$ on $\overline{S}$ with points marked by $\tr(\cB)$:
		$$\int_{(S,\partial S)}(\cA,\cB)\simeq \int_{\overline{S},\{x_i\}} (\cA,\tr(\cB)).$$
\end{enumerate}
\end{proposition}

The proposition (in the case of representations of a quantum Borel or parabolic subgroup) allows us in particular to define parabolic versions of quantum character varieties, quantizing moduli of parabolic local systems.

\section{Braided module categories and surfaces with marked points}\label{sec:BraidedModule}

In this section, we identify the possible markings of points in factorization homology of marked surfaces in terms of explicit algebraic data called braided module categories.  To begin, we recall that the factorization homology of the annulus inherits a natural tensor structure, coming from ``stacking annuli":

\begin{figure}[h]
\includegraphics[width=4in]{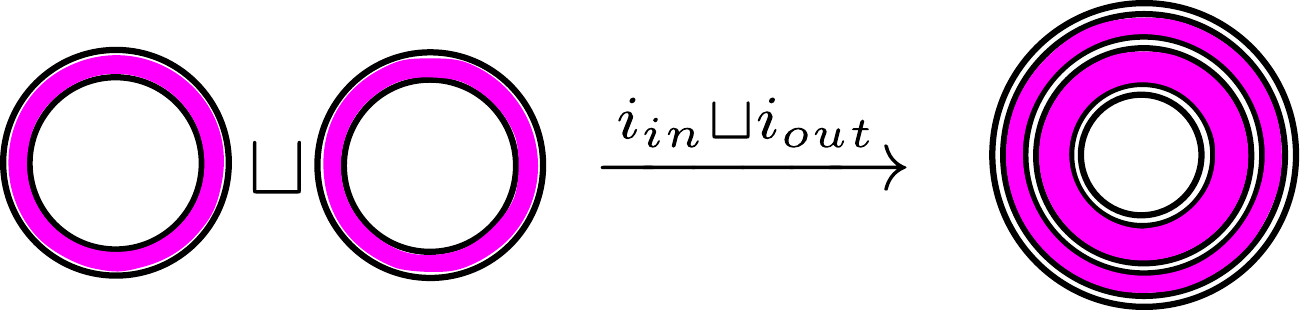}
\caption{The inclusion of two annuli into a third induces the stacking tensor structure on $\int_{Ann}\cA$ by functoriality.}\label{stacking-picture}
\end{figure}

\begin{definition} The \emph{stacking tensor product} on $\int_{Ann}\cA$, denoted $M,N\leadsto M\odot N$, is:
$$T_{St}:= (i_{in} \sqcup i_{out})_*: \int_{Ann}\!\!\!\!\!\cA\,\,\, \bt\,\,\, \int_{Ann}\!\!\!\!\!\cA \longrightarrow \int_{Ann}\!\!\!\!\!\cA,$$
where $i_{in}$ and $i_{out}$ are as depicted below.
\end{definition}

\begin{remark}
This same category carries a second tensor product, the \emph{pants}, or \emph{convolution} tensor product. Let $Pants$ denote a twice punctured disc, let $i_1$ and $i_2$ denote the inclusion of an annular boundary around the two punctures, and let $i_{out}$ denote the annular boundary around the outside of the disc.  Then we define:
$$T_{Pants}:= i_{out}^* \circ ( i_1\sqcup i_2)_*: \int_{Ann}\!\!\!\!\!\cA \,\,\,\bt\,\,\, \int_{Ann}\!\!\!\!\!\cA \longrightarrow \int_{Pants}\!\!\!\!\!\!\!\cA \longrightarrow \int_{Ann}\!\!\!\!\!\cA.$$

One can show following~\cite{Bruguieres2008} that this coincides with the monoidal structure of the Drinfeld center of $\cA$. An illustrative example for comparing these two structures is the case $\cA=\Rep G$.  Then we have
$$\int_{Ann}\!\!\!\!\Rep G \simeq \QCoh(\frac{G}{G}),$$ and the stacking tensor product is the pointwise tensor product of quasi-coherent sheaves on the stack $\frac{G}{G}$, while the pants tensor product is the convolution along the diagram,
$$\frac{G}{G}\times\frac{G}{G} \xleftarrow{quot.} \frac{G\times G}{G} \xrightarrow{mult.} \frac{G}{G}.$$
The convolution tensor structure does not play a role in this paper, as it is the stacking tensor product which features in the excision axiom.
\end{remark}

There is a natural operadic notion of module for an $E_n$-algebra $\cA$, the notion of $E_n$-module (which for $n=1$ recovers the notion of {\em bimodule} rather than module over an associative algebra), see~\cite{Lurie2009,Francis2013} and~\cite{Ginot2014} for a review. Intuitively an $E_n$-module is an object placed at the origin in $\RR^n$ and admitting operations from insertions of $\cA$ at disjoint points (or disks). This notion appears naturally in factorization homology, as the data required to extend the theory to manifolds with marked points. There's a monadic description of $E_n$-modules over $\cA$ as left modules over an $E_1$-algebra $U(\cA)$, the enveloping algebra of $\cA$, which is in turn readily identified with the factorization homology $\int_{Ann}\cA$ of $\cA$ on the complement of the origin. We take this latter notion as our definition and refer to~\cite{Francis2013,Ginot2014} for comparisons with the operadic notion. 
We will then give it a more algebraic description by reconciling it with the notion of a ``braided module category"~\cite{Brochier2013,Enriquez2008}, in the case $\cC=\Rex$.

\begin{definition-proposition}[\cite{Ginot2014}]
An $E_2$-module for an $E_2$-algebra $\cA$ is a right module over the annulus category $\int_{Ann} \cA$, with tensor structure $T_{St}$.
\end{definition-proposition}

In our setting, $E_2$-algebras are identified with the notion of braided tensor categories; in the same spirit, we first show that the notion of $E_2$-modules over braided tensor categories coincides with the notion of braided module categories defined as follows:
\begin{definition}[\cite{Brochier2013}]\label{def:BraidedModule}
	Let $\cA$ be a braided tensor category, with braiding $\sigma$. A (strict\footnote{This assumption is made only to simplify the exposition: the non-strict axioms simply involve inserting associators in the obvious places in equations \eqref{eq:octagon} and \eqref{eq:DMcat}.}) \emph{braided module category} over $\cA$ is a (strict) right $\cA$-module category $\cM$ equipped with a natural automorphism, $E$, of the action bifunctor
	\[
\ot : \cM \times \cA \longrightarrow \cM
	\]
	satisfying the following axioms for all $M \in \cM$, $X,Y \in \cA$:
	\begin{equation}\label{eq:octagon}
		E_{M\ot X,Y}=\sigma_{X,Y}^{-1}E_{M,Y}\sigma_{Y,X}^{-1}
	\end{equation}
	and
	\begin{equation}\label{eq:DMcat}
		E_{M,X\ot Y}=E_{M,X}E_{M\ot X,Y}\sigma_{Y,X}\sigma_{X,Y}.
	\end{equation}
\end{definition}

\begin{remark}
	In the presence of~\eqref{eq:octagon}, equation~\eqref{eq:DMcat} is equivalent to the following version of the Donin--Kulish--Mudrov equation~\cite{Donin2003a}
	\begin{equation}\label{eq:DM}
		E_{M,X\ot Y}=E_{M,X}\sigma_{X,Y}^{-1}E_{M,Y}\sigma_{X,Y}.
	\end{equation}

This in turn implies the reflection equation, and hence that braided module categories give rise to representations of the braid group of the annulus. In fact equation~\eqref{eq:octagon} alone implies the reflection equation, but both are needed to describe the full $E_2$-module structure.
\end{remark}

\begin{remark}\label{different framings}
Axioms \eqref{eq:octagon}, \eqref{eq:DMcat} differ slightly from those in~\cite{Brochier2013}:  the latter characterizes modules over the factorization homology of $S^1\times \RR$ equipped with the \emph{product/cylinder} framing, while $E_2$-modules are rather characterized as modules over the factorization homology of $S^1\times \RR$ equipped with the \emph{blackboard} (annulus) framing.  As there are an integer's worth of framings on $S^1\times \RR$, hence there 
are an integer's worth of alternative notions of braided module category.

When $\cA$ is balanced, however the corresponding axioms are equivalent:  indeed, if $\theta$ is the balancing and $E$ satisfies the axioms stated above, then a straightforward computation shows that $E^{-1}(\id \ot \theta^{-1})$ satisfies the axioms of~\cite{Brochier2013}.  We note that a similar phenomenon has appeared in the definition of the elliptic double of~\cite{Brochier2014}, where there are many possible definitions, which coincide in the case of a ribbon Hopf algebra.  Here, in order to simplify the presentation, we stick to the balanced assumption and do not differentiate between these different framings.
\end{remark}

In order to identify the different notions of module category we will use a different description of the annulus than that which features in~\cite{Ben-Zvi2015}. Rather than cutting it into two half-annuli, as we did there, we will make a single vertical cut at the top of the annulus and see the annulus as being obtained by self-gluing along this cut (see Figure~\ref{fig:annulus-cut}).
\begin{figure}[h]\label{fig:annulus-cut}
\includegraphics[width=5in]{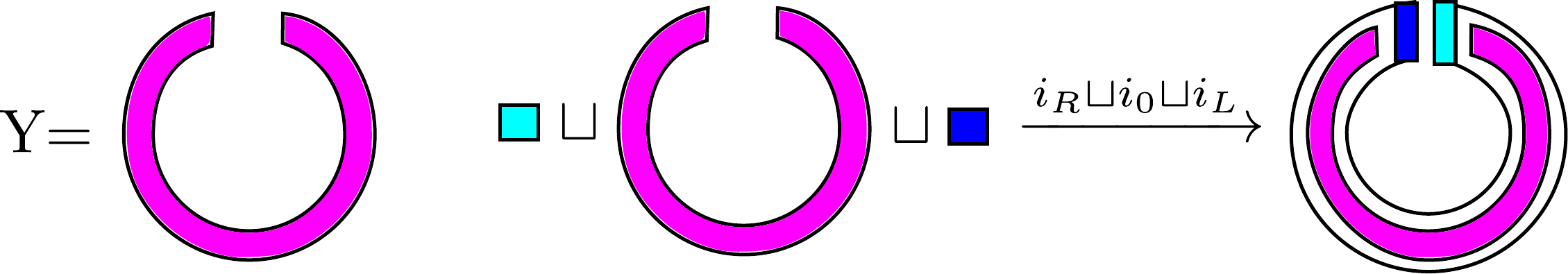}
\caption{Left:  the stratified manifold $Y$.  Right: the inclusions of discs inducing the $\cA$-bimodule structure on $\int_Y\cA$.  The annulus is obtained by gluing along the boundary intervals.} \label{fig:annulus-action}
\end{figure}

Hence let $Y$ be the manifold represented in Figure~\ref{fig:annulus-cut}.
We stress the fact that for the gluing to make sense we need to regard it as a stratified manifold, and as such it is not equivalent to the standard framed disc. Yet, its interior $\mathring{Y}$, i.e. the manifold obtained from $Y$ by forgetting the stratification, is equivalent as a framed manifold to the standard disc, hence we have an equivalence of underlying categories,
\[
	\int_Y \cA \simeq \int_{\mathring{Y}} \cA \simeq \cA
\]
but the $\cA$-actions are different from the standard one. The $\cA$-bimodule action on $\int_{Y}\cA$ is induced by the embedding $i_R\sqcup i_0 \sqcup i_L: \RR^2 \sqcup \mathring{Y} \sqcup \RR^2 \rightarrow Y$ depicted in the left hand side of Figure~\ref{fig:annulus-action}, i.e. it is the functor,
$$(i_R\sqcup i_0 \sqcup i_L)_*: \cA\bt\left(\int_Y\cA\right)\bt\cA \to \int_Y\cA,$$
\begin{figure}[h]\label{fig:twistedModule}
\includegraphics[width=5in]{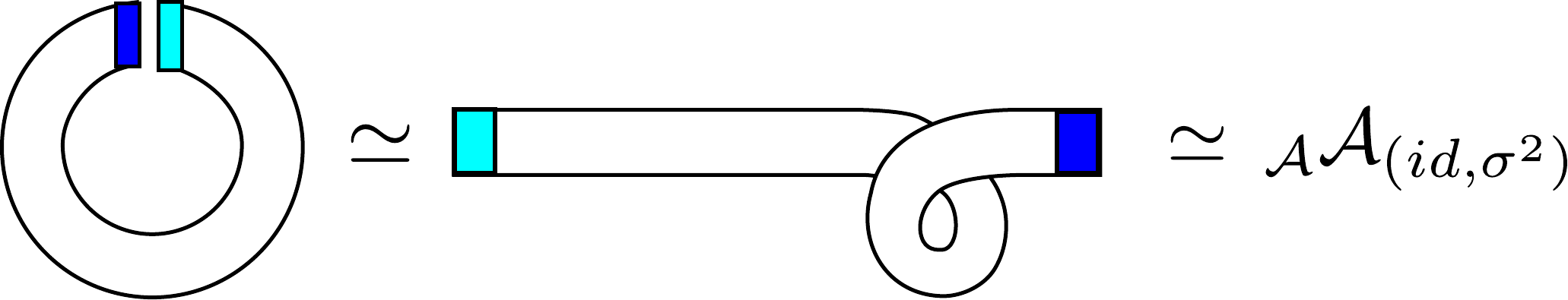}
\caption{The $\cA$-bimodule $\int_Y\cA$ is obtained from the regular bimodule by precomposing the right action by the tensor functor $(id,\sigma^2):\cA\rightarrow \cA$.}
\end{figure}

The double braiding $\sigma^2$ induces a non-trivial tensor structure on the identity functor of $\cA$.  We denote by $_\cA\cA_{(id,\sigma^2)}$ the twist of the regular bimodule category by this auto-equivalence, i.e. the $\cA$-bimodule category whose underlying left module is the regular left $\cA$-module, but where the right action (in particular, its associativity constraint) is precomposed with $(id,\sigma^2)$.   In Figure~\ref{fig:twistedModule}, we demonstrate an equivalence between $\int_Y\cA$ and this twisted bimodule.

To recover the factorization homology of the annulus from the bimodule $\int_Y\cA$, we need to recall the notion of balanced functors of bimodules, and the resulting notion of the trace of a bimodule.  To this end, let $\cM$ be an $\cA-\cA$-bimodule category. Then a functor $F:\cM\rightarrow \cE$ is called balanced if there is a natural isomorphism
\[
	F(m\ot X)\rightarrow F(X \ot m)
\]
satisfying a natural coherence condition (see e.g.\cite{Fuchs2014}).

\begin{definition}\label{defi:trace}
	The \emph{trace} $\tr\cM$, of an an $\cA$-$\cA$-bimodule category $\cM$, is the $\Rex_{\K}$-category defined uniquely by the natural equivalence,
$$\Rex_{\K}[\tr\cM, \cE] \simeq \operatorname{Bal}_\cA(\cM,\cE).$$
\end{definition}
\begin{remark}
If $\cM,\cN$ are right and left $\cA$-module categories, then $\cM \boxtimes \cN$ is an $\cA$-bimodule and clearly 
\[
	\tr (\cM \boxtimes \cN)\simeq \cM \boxtimes_{\cA} \cN.
\]
On the other hand, an $\cA$-bimodule is the same as an $\cA \boxtimes \cA^{rev}$-module and one can show that
\[
	\tr \cM \simeq \cA \boxtimes_{\cA\boxtimes \cA^{rev}} \cM
\]
where $\cA$ is given its natural $\cA$-bimodule structure (i.e., we recover the standard notion of Hochschild homology of a bimodule).
\end{remark}
We have:
\begin{lemma}
The category $\int_{Ann}\cA$ is the trace in the sense of Definition~\ref{defi:trace} of the bimodule $_{\cA}\cA_{(id,\sigma^2)}$. 
\end{lemma}
\begin{proof}
	It is clear that constructing the annulus by self-gluing from the manifold showed on Figure~\ref{fig:annulus-cut} is the same as gluing both ends to a disk with two marked intervals and its standard framing, which corresponds to the regular bimodule $\cA$. Hence by the excision property we have:
	\[
		\int_{Ann}\cA\simeq _{\cA}\cA_{(id,\sigma^2)}\bt_{\cA \bt \cA^{rev}} \cA\simeq \tr \left(_{\cA}\cA_{(id,\sigma^2)}\right).
		\]
\end{proof}
We can now prove:
\begin{proposition}\label{prop:AnnulusAction}
	Let $\cA$ be a braided tensor category, $\cB$ a tensor category. Then tensor functors $\int_{Ann} \cA\rightarrow \cB$ are naturally identified with pairs $(F,\nu)$ where $F$ is a tensor functor from $\cA$ to $\cB$, and $\nu$ is a natural automorphism of $F$ satisfying\footnote{Here and in the proof we abuse notation and write $\sigma_{X,Y}$ instead of $F(\sigma_{X,Y})$.}:
	\begin{equation}\label{eq:AnnAction1}
		\id_X \ot \nu_{Y}=\sigma_{X,Y}^{-1} (\nu_{Y}\ot \id_X)\sigma_{Y,X}^{-1}
	\end{equation}
	and
	\begin{equation}\label{eq:AnnAction2}
		\nu_{X\ot Y}=(\nu_X \ot \nu_Y)\sigma_{Y,X}\sigma_{X,Y}
	\end{equation}
\end{proposition}

\begin{proof}
Denote by $\rtr$ and $\ltr$ the left and right $\cA$-actions on itself described above.	The description of $\int_{Ann}\cA$ as a coequalizer of those actions implies that functors out of it to a target category $\cB$ are naturally identified with functors $F:\cA\rightarrow \cB$ equipped with a natural isomorphism
	\[
		\eta_{X,Y}:F(X \ot Y)=F(X \ltr Y) \rightarrow F(Y\ot X)=F(Y\rtr X)
	\]
	satisfying the following coherence condition
	\[
		\sigma_{Z,Y}\sigma_{Y,Z}\eta_{X,Y\ot Z}=\eta_{Z\ot X,Y}\eta_{X\ot Y,Z}
	\]
	
	In order to characterize the monoidal structure on $\int_{Ann}\cA$, we need to upgrade this coequalizer as the coequalizer of a diagram in tensor categories. Recall that any tensor functor $F:\cA \rightarrow \cB$ turns $\cB$ into a left $\cA$-module using the following composition
	\[
		\cA \boxtimes \cB \xrightarrow{F\boxtimes \id} \cB \boxtimes \cB\xrightarrow{m} \cB
	\]
	where $m$ is the multiplication of $\cB$ and the associativity constraint of the module structure is given by the monoidal structure on $F$. This also turn $\cB$ into a right module by using the opposite multiplication instead.
	We note that both $\cA$ actions at hand are of this form: the right action is just induced by the identity functor and the left action by the identity functor with monoidal structure given by the double braiding. Since the multiplication of $\cA$ also carries a natural monoidal structure, this turns the maps involved in the defining diagram of $\int_{Ann}\cA$ into tensor functors.
	
	Now, (strict) monoidal functors out of $\int_{Ann}\cA$ to a monoidal category $\cB$ can be characterized as strict monoidal functors $F:\cA\rightarrow \cB$ equipped with a cyclic structure $\eta$ as above for which $\eta$ is monoidal. This leads to the following identity:
\[
	\eta_{X\ot W, Y\ot Z}\sigma_{Y,W}=\sigma_{Z,Y}\sigma_{Y,Z}\sigma_{X,Z}\eta_{X,Y}\eta_{W,Z}
\]
The coherence condition implies that $\eta_{X,1}=\id_X$. Hence define a natural automorphism of $F$ by $\nu_X:=\eta_{1,X}$. Setting $X=Z=1$ gives
\[
	\eta_{W,Y}\sigma_{Y,W}=\nu_Y
\]
so that $\eta$ can be recovered from $\nu$. For $W=Y=1$ this gives
\[
	\eta_{X,Z}=\sigma_{X,Z} \nu_Z.
\]
Together those equations implies~\eqref{eq:AnnAction1}.
Finally, setting $X=W=1$ leads to
\[
	\nu_{X\ot Z}=\sigma_{Z,Y}\sigma_{Y,Z}\nu_X\nu_Z.
\]
which up to relabelling is exactly~\eqref{eq:AnnAction2}.
\end{proof}

Hence, we have:

\begin{theorem}\label{thm:e2mod}
Let $\cA$ be a braided tensor category. Then $E_2$-modules over $\cA$ are naturally identified with $\cA$-braided module categories in the sense of \cite{Brochier2013}.
\end{theorem}

\begin{proof}
	A right module $\cM$ over $\int_{Ann} \cA$ is characterized by the tensor functor $\int_{Ann}\cA\rightarrow \End(\cM)$ given by $X\mapsto - \ot X$. It means that $\cM$ has to be an $\cA$-module and that one has to provide the functor $\cA \rightarrow \End(\cM)$ with a natural automorphism satisfying the axioms above. This is straightforward to check that this turns $\cM$ into a braided module category.
\end{proof}

\subsection{The oriented case}

Recall that factorization homology of framed surfaces descends to an invariant of \emph{oriented} surfaces provided $\cA$ is balanced, i.e. that $\cA$ is endowed with an automorphism $\theta: \id_\cA\to\id_\cA$ of the identity functor of $\cA$, satisfying the coherence relation,
$$\theta_{V\ot W} = \sigma_{W,V}\circ\sigma_{V,W}\circ (\theta_V\otimes \theta_W).$$

In the same way, one can ask what additional structure is needed on a braided module category $\cM$ over a balanced braided monoidal category $\cA$ in order to obtain invariants of oriented marked surfaces. We have:

\begin{theorem}\label{thm:balancedModule} Let $\cA$ denote a braided tensor category and let $\cM\in \Rex$.
\begin{enumerate}
	\item Given a braided module category structure on $\cM$, the additional structure of a $\Disk^2_{or}$-module extending $\cA$ and $\cM$ consists, first of all of a balancing on $\cA$, and secondly of a ``balancing automorphism" $\phi_\cM:\id_\cM\to\id_\cM$ of the identity functor on $\cM$, satisfying the coherence:
\begin{equation}\label{eq:balancedModule}
\phi_{M\ot X}=E_{M,X}\circ(\phi_M\ot\theta_X)
\end{equation}
We refer to a braided module category $\cM$, equipped with a $\phi_\cM$, as a ``balanced braided module category".

\item Suppose that $\cA$ is a balanced braided tensor category, and that $\cM$ is a braided module category for $\cA$.  Then $\cM$ admits a canonical structure of a balanced braided module category.
\end{enumerate}
\end{theorem}
\begin{proof}
The first part is clear from the picture: $\phi$ is the automorphism of the identity functor of $\cM$ induced by the rotation of a marked disc inside a larger one.

If $\cA$ is balanced, then $\int_{Ann}\cA$ is independent of the framing and in particular comes equipped with an automorphism of the identity functor $\psi$ coming from the loop, in the space of oriented diffeomorphisms of the annulus, which rotates $\theta$ degrees about the origin, for $\theta\in [0,2\pi]$. Let $F=F_{bd}:\cA\longrightarrow \int_{Ann}\cA$ be band tensor functor defined in Figure~\ref{annulus-tensor-functor}, and $\nu$ the automorphism of $F$ as in Proposition~\ref{prop:AnnulusAction}. It is again a direct check that for all $X\in \cA$,
\begin{equation}\label{eq:Balancing}
	\nu_X=\psi_{F(X)}\theta_X^{-1}\psi_{\un_{\int_{Ann}\cA}}^{-1}
\end{equation}

where we identify $F(X)$ with $F(X)\ot\un_{\int_{Ann}\cA}$.

By Proposition~\ref{prop:AnnulusAction}, a braided module structure $E$ on $\cM$ is the same as a factorization of the functor
\[
	\cA\longrightarrow \End(\cM)
\]
as 
\[
	\cA\longrightarrow \int_{Ann} \cA \longrightarrow \End(M)
\]
and $E$ is defined as the image of $\nu$ through the second functor. Taking the image of~\eqref{eq:Balancing} through the second functor gives the desired balanced structure on $\cM$.

\end{proof}

\begin{remark} We note that the balancing asserted in (2), while canonical, is not unique.  There may be many different balancings on a given braided module category.
\end{remark}

\section{Reconstruction theorems}\label{reconstruction-section}
In~\cite{Ben-Zvi2015}, we developed a framework to describe the gluing of surfaces along intervals in their boundary, using monadic techniques:  the factorization homology for a surface with a marked interval in the boundary obtained the structure of an $\cA$-module category, and we realized these $\cA$-module categories as the categories of modules for explicit algebras $A_S$ in $\cA$.  In this section, we will develop some of the general algebraic tools we will need to glue along circles, rather than intervals, in the boundary of a surface.  Namely, the pointwise tensor structure from Figure~\ref{annulus-tensor-functor} defines a dominant tensor functor from $\cA$ to $\int_{Ann}\cA$, and we wish to apply monadic techniques to understand the resulting structures algebraically.

\textbf{Assumptions}  We will typically work with the $(2,1)$-categories $\Rex$ of essentially small finitely cocomplete $\K$-linear categories, right exact functors and their natural isomorphisms, and the $(2,1)$ category $\LFP$ of compactly generated presentable categories with compact and cocontinuous functors and their natural isomorphisms.  These each carry the so-called Kelly-Deligne tensor product, and are in fact equivalent to one another as symmetric monoidal $(2,1)$-categories:  the functors $\ind$ and $\comp$, of $\ind$-completing a $\Rex$ category to a $\LFP$ category, and taking the $\Rex$ subcategory of compact objects of an $\LFP$ category, are mutually inverse equivalences.  For orientation, let us remark that small abelian categories are in particular $\Rex$, while Grothendieck abelian categories are in particular $\LFP$.

By a tensor (or braided tensor) category in $\Rex/\LFP$, we will mean simply an $E_1$- (or $E_2$-) algebra $\cA$ in $\Rex/\LFP$.  We will typically assume that $\cA$ is \emph{rigid}, i.e. that all (compact) objects are left and right dualizable.  This categorical framework is discussed in detail in \cite{Ben-Zvi2015}, to which we refer the reader for complete definitions. 

\subsection{Reconstruction from tensor functors}
Let $\cA$ and $\cB$ be tensor categories in $\LFP$ and $F:\cA\rightarrow \cB$ be a tensor functor.  Suppose that $\cA$ is rigid, and assume the $\un_\cB$ is a pro-generator for the $\cA$-module structure on $\cB$ induced by $F$\footnote{Note that it is a pro-generator for the left $\cA$-action induced by $F$ if and only if it is for the right one.}. The definition of a pro-generator implies that $F=\act_{\un_\cB}$ has a cocontinuous right adjoint which is also faithful, i.e. $F$ is dominant. Hence, by applying Theorem~4.5 from~\cite{Ben-Zvi2015} to $\cB$ as a $\cA$-module category, we see that $\cB$ admits a simultaneous description, both as right-, and as left- $\underline{\End}(\mathbf{1}_\cB)$-modules in $\cA$, where we recall that $\underline{\End}(\mathbf{1}_\cB)\cong F^R(\mathbf{1}_\cB)$.  In this section, we extend this description to encompass the tensor structure on $\cB$ as well.

\begin{theorem}\label{reconstruction}  We have:
\begin{enumerate}
\item Under this description $F$ identifies with the free module functor $V\mapsto \underline{\End}(\mathbf{1}_\cB)\ot V$, and $F^R$ identifies with the forgetful functor back to $\cA$, where $\underline{\End}$ means internal endomorphism of $\cB$ as an $\cA$-module.
\item  The right and left $\underline{\End}(\mathbf{1}_\cB)$-module structures on $F^R(b)$, for any $b\in\cB$, are related by the isomorphisms:
\begin{equation}\underline{\End}(\mathbf{1}_\cB)\ot F^R(b) \cong F^R(\mathbf{1}_\cB\ot b) \cong F^R(b \ot \mathbf{1}_\cB) \cong F^R(b) \ot \underline{\End}(\mathbf{1}_\cB).\label{eqn-exchanges}\end{equation}
These actions commute, and we obtain a faithful tensor functor,
$$\widetilde{F^R}:\cB\to \underline{\End}(\mathbf{1}_\cB)\textrm{-bimodules in $\cA$}.$$
\item Moreover $\widetilde{F^R}$ becomes a tensor functor, when we equip the category of $\underline{\End}(\mathbf{1}_\cB)$-bimodules in $\cC$ with the relative tensor product of bimodules:
$$M\underset{\underline{\End}(\mathbf{1}_\cB)}{\ot} N:= \operatorname{colim} \left(M\ot \underline{\End}(\mathbf{1}_\cB)\ot N \underset{\operatorname{act}_L}{\overset{\operatorname{act}_R}{\rightrightarrows}} M\ot N\right),$$
\end{enumerate}
\end{theorem}

\begin{proof}
For Claim (1), we need a natural isomorphism $F^RF(V)\cong \underline{\End}(\mathbf{1}_\cB)\ot V$.
Applying the tensor structure and the adjunction counit, we obtain natural isomorphisms:
\begin{align*}\Hom(C,F^RF(V)) &\cong \Hom(F(C),F(V))\\
&\cong \Hom(^*F(V)\ot F(C),\mathbf{1}_\cB)\\
&\cong \Hom(F(^*V)\ot F(C),\mathbf{1}_\cB)\\
&\cong \Hom(F(^*V\ot C),\mathbf{1}_\cB)\\
&\cong \Hom(^*V\ot C,F^R(\mathbf{1}_\cB))\\
&\cong \Hom(C,F^R(\mathbf{1}_\cB)\ot V).
\end{align*}
Hence by Yoneda's lemma, we have a natural isomorphism $F^RF(V)\cong F^R(\mathbf{1}_\cB)\ot V$, as desired.

Claim (2) is clear: the $\underline{\End}(\mathbf{1}_\cB)$-actions on $F^R(b)$ are given in terms of the adjunction data $(F,F^R)$; the isomorphisms \eqref{eqn-exchanges} are natural in $b$, and interchange left and right modules in the adjunction.

For Claim (3), we note that the category of $F^R(\mathbf{1}_\cB)$-bimodules is generated under colimits by the free bimodules, so that if we have such an isomorphism for the free bimodules, it necessarily induces the same natural isomorphism for all bimodules.   Hence, we may restrict to the case that $M$ and $N$ are of the form $F(m)$, $F(n)$, for some $m,n \in \cB$.  In that case, applying $F^RF$ to the obvious idenity,
$$m\ot n = \operatorname{colim} \left(m\ot \mathbf{1}_\cA\ot n \underset{\operatorname{act}_L}{\overset{\operatorname{act}_R}{\rightrightarrows}} m\ot n\right),$$
gives $$F^R(F(m)\ot F(n)) = F^R\left(\operatorname{colim}\left(F(m)\ot \mathbf{1}_\cB\ot F(n) \underset{\operatorname{act}_L}{\overset{\operatorname{act}_R}{\rightrightarrows}} F(m)\ot F(n)\right)\right).$$
While we cannot commute $F^R$ past the colimit, we have the canonical comparison map:
\begin{multline*}
\operatorname{colim}\left(F^R(F(m)\ot \mathbf{1}_\cB\ot F(n)) \underset{\operatorname{act}_L}{\overset{\operatorname{act}_R}{\rightrightarrows}} F^R(F(m)\ot F(n))\right)\\ \to F^R\left(\operatorname{colim}\left(F(m)\ot \mathbf{1}_\cB\ot F(n) \underset{\operatorname{act}_L}{\overset{\operatorname{act}_R}{\rightrightarrows}} F(m)\ot F(n)\right)\right).
\end{multline*}
Moreover, we have $F^R(F(m)\ot \mathbf{1}_\cB\ot F(n))\cong F(m)\ot F^R(\mathbf{1}_\cB)\ot F(n)$, because the right adjoint $F^R$ to the $\cA$-bimodule functor $F$ is canonically a $\cA$-bimodule functor whenever $\cA$ is rigid.
Finally, we see that the comparison map is in fact an isomorphism, because $F$ is the free bimodule functor.
\end{proof}

Having described the tensor structure on $\cB$ monadically through $\cA$, we now turn to describing algebra objects, and hence module categories, for $\cB$, monadically in terms of $\cA$.  We have:

\begin{proposition}\label{momentmaps} The equivalence $\cB\simeq \underline{\End}(\mathbf{1}_\cB)\modu_{\cA}$ extends to an equivalence:
$$\left\{\textrm{Algebras in $\cB$}\right\} \simeq \left\{\begin{array}{l}\textrm{Algebras in $\cA$, equipped with an}\\\textrm{algebra homomorphism from $\underline{\End}(\mathbf{1}_\cB)$}\end{array}\right\}.$$
\end{proposition}
\begin{proof}
Given an algebra object $b$ in $\cB$, its image $F^R(b)$ in $\cA$ receives a canonical algebra homomorphism $F^R(\mathbf{1}_\cB)\to F^R(b),$ induced by the unit homomorphism $\mathbf{1}_\cB\to b$, through the lax tensor structure on $F^R$.  This provides a functor in the forward direction. 

Conversely, given an algebra in $\cB$ equipped with a homomorphism from $F^R(\mathbf{1}_\cB)$, we make it an algebra in the category of $F^R(\mathbf{1}_\cB)$-modules via this homomorphism; this clearly endows it with the structure of an algebra object in $F^R(\mathbf{1}_\cB)$-modules.  This provides a functor in the reverse direction.

The two functors we have constructed are mutually inverse for tautological reasons: the unit homomorphism in the forward direction, and the $F^R(\mathbf{1}_\cB)$-module structure in the reverse direction are simply equivalent data.

\end{proof}

Moreover, we can describe the $\cB$ action on one of its module categories monadically through $\cA$:

\begin{theorem}\label{inducedactions}
Let $\cM$ be a $\cB$-module category, and let $M\in\cM$ be a progenerator for the induced $\cA$-action.  Recalling the equivalences of $\cA$-modules,
$$\cB \simeq \operatorname{mod}_{\cA}-\underline{\End}(\mathbf{1}_\cB), \qquad \cM \simeq \underline{\End}(M)\modu_{\cA},$$
and the canonical algebra homomorphism,
$$\rho:\underline{\End}(\mathbf{1}_\cB)\to \underline{\End}(M),$$
we have:
\begin{enumerate}
	\item  The action of any $b\in\cB$ on any $N\in \cM$ is given by:
$$N\bt b \mapsto b\underset{\underline{\End}(\mathbf{1}_\cB)}{\ot} N.$$
	\item The $\cA$-progenerator $M$ is also a progenerator for the $\cB$-action.
	\item The functor $\cM=\underline{\End}_{\cB}(M)\modu_{\cB}\rightarrow \underline{\End}_{\cA}(M)\modu_{\cA}$ induced by $\rho$ is an equivalence of $\cB$-module categories.
\end{enumerate}
\end{theorem}

\begin{proof}
The first claim is a direct application of Proposition \ref{reconstruction}.  For the second claim, we recall that, to say $M$ is a pro-generator, is to say that the functor $act_M^R:\cM\to\cB$ is conservative, and co-continuous.  It follows from our assumptions that the composite functor $F^R\circ\act_M^R$ is conservative and co-continuous, as the action of $\cA$ on $\cM$ is obtained by pull-back, and $M$ is assumed to be a $\cA$ pro-generator.  By assumption $F^R:\cB\to\cA$ is conservative and co-continuous.  It is an easy exercise that conservative and co-continuous functors themselves reflect conservativity and co-continuity.  The third claim follows from monadicity for base change, Corollary~4.11 in~\cite{Ben-Zvi2015}.
\end{proof}
And finally, we can compute relative tensor product of module categories, simply as categories of bimodules:

\begin{corollary}\label{cor:bimoduleGluing}
Fix $\cA$, $\cB$ and $F$ as in Proposition \ref{momentmaps}.  Let $\cM$ and $\cN$ be left and right $\cB$-module categories, with $\cA$-progenerators $m$ and $n$, respectively.  Then we have an equivalence of categories,
$$\cM\underset{\cB}{\bt}\cN \simeq (\underline{\End}(m)-\underline{\End}(n))\operatorname{-bimod}_{\underline{\End}(\mathbf{1}_\cB)\modu_{\cA}}.$$
\end{corollary}
\begin{proof} This is just an application of Theorem~4.9 from~\cite{Ben-Zvi2015}, and Proposition \ref{inducedactions} above.\end{proof}

\subsection{Reconstruction for the annulus category}\label{annulus-category-section}

In this section, we apply the reconstruction techniques of the preceding section to the setting of factorization homology for braided module categories, to give a proof of Theorem \ref{quantum-moment-map-theorem}. To begin with, we consider the inclusion $j_0: D^2\to Ann$, given by including a small disk $D^2$ into an annulus along some small band, as depicted in Figure \ref{annulus-tensor-functor}.

\begin{figure}[h]\label{annulus-tensor-functor}
\includegraphics[height=2in]{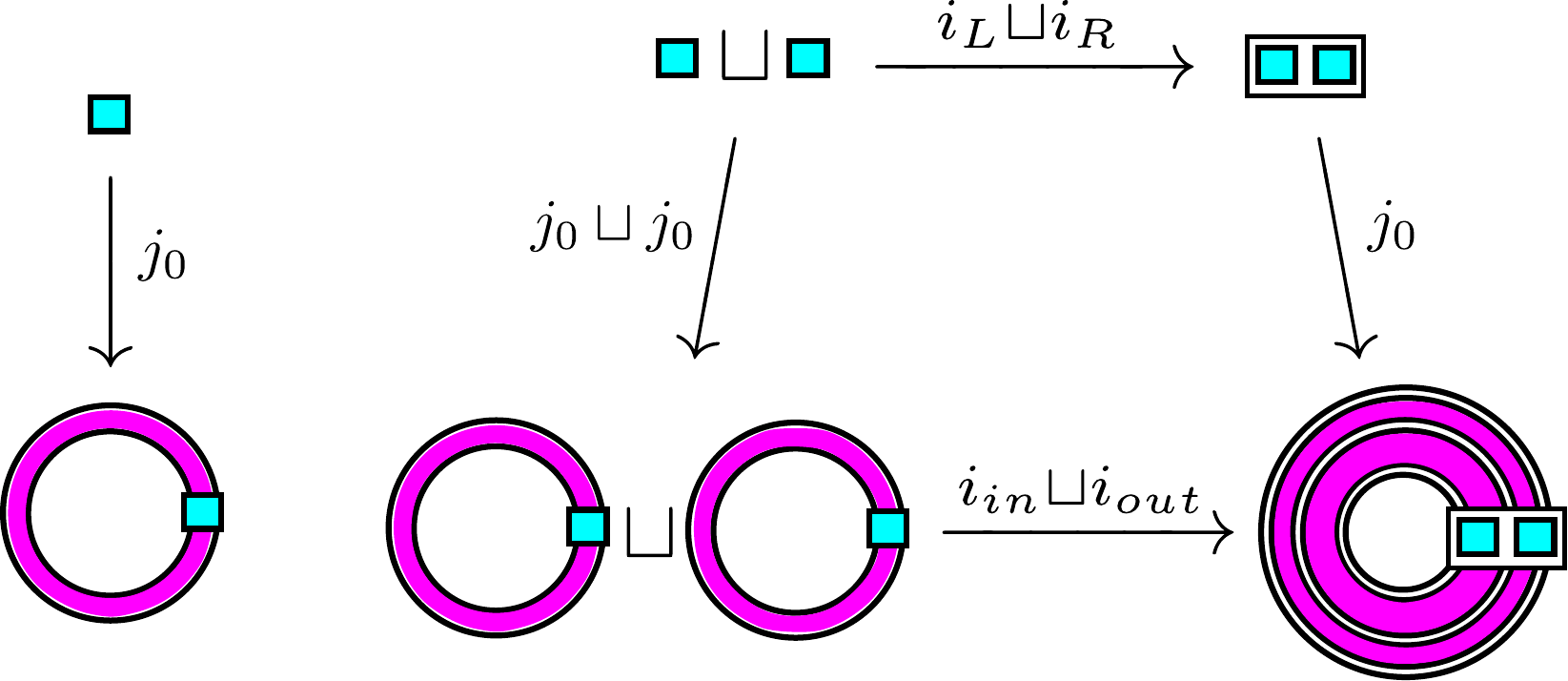}
\caption{Left: the tensor functor $F_{bd}:\cA\to\int_{Ann}\cA$ is induced by an inclusion of a disc into a small radial band in the annulus.  Right: the tensor structure induced by a commutative diagram, up to isotopy.}
\end{figure}

\begin{definition}\label{annulus-tensor-functor-def} The ``band" tensor functor $F_{bd}:\cA\to\int_{Ann}\cA$ is the functor $(j_0)_*$, induced by functoriality of factorization homology, with tensor structure induced by the commuting-up-to-isotopy diagram of inclusions, depicted in Figure \ref{annulus-tensor-functor}.
\end{definition}

With all of this framework in place, we turn to the proof of Theorem \ref{quantum-moment-map-theorem}.

\begin{proof}[Proof (of Theorem \ref{quantum-moment-map-theorem})]
We shall apply the results of Section \ref{reconstruction-section} to the special case that $\cA$ is a braided tensor category in $\Rex$, and
$$\cD=\int_{Ann}\cA\simeq \oa\modu_\cA, \qquad F=F_{bd}.$$

It follows from Theorem 4.16 of \cite{Ben-Zvi2015} that $\cO_\cA$ is a pro-generator for the $\cA$-action, and that $F_{bd}\cong \act_{\cO_\cA}$.  Hence, by Theorem~\ref{reconstruction}, $F_{bd}$ is naturally isomorphic to the free module functor $V\mapsto \oa\ot V$, and $F_{bd}^R:\int_{Ann}\cA \to \cA$ identifies simply with the forgetful functor from $\oa\modu_\cA\to\cA$.

For (1), the quantum moment map is that constructed in Proposition~\ref{momentmaps}.  For (2), Theorem~4.5 from \cite{Ben-Zvi2015}, combined with Proposition~\ref{momentmaps}, give equivalences:
$$\cM \simeq \underline{\End}_{\cA}(M)\modu_{\cA} \simeq \underline{\End}_{\cA}(M)\modu_{\int_{Ann}\cA},$$
where the latter is equipped with the algebra structure coming from the quantum moment map.  Finally, (3) follows the same proof as Part (3) of Proposition \ref{reconstruction}.
\end{proof}

Unpacking the isomorphism of Theorems \ref{reconstruction} and \ref{inducedactions} in this case, we have the following corollaries:

\begin{corollary}
Any left $\oa$-module has a canonical right module structure. The left and right action are related using the ``field goal'' transform $\tau_m:\oa\ot m  \rightarrow m\ot \oa$:
\[
	\begin{tikzpicture}
\braid[
 line width=2pt,
 style strands={2,3}{blue},
] (braid2) at (0,0) s_1^{-1}s_2 ;
\draw (braid2-1-e) node[below]{$m$};
\draw (braid2-1-s) node[above]{$m$};
\draw (braid2-3-e) node[below left]{$\oa$};
\draw (braid2-3-s) node[above left]{$\oa$};
\draw (braid2) node[left=1.5cm]{$\tau_m=$};
\end{tikzpicture}.
\]
For any $M, N\in \int_{Ann} \cA\simeq \oa\modu_\cA$, the pointwise tensor product $M\odot N$ is given by the relative tensor product,
$$M\odot N \cong M\underset{\oa}{\ot} N:= \operatorname{colim} \left(M\ot\oa\ot N \underset{\operatorname{act}_L}{\overset{\operatorname{act}_R}{\rightrightarrows}} M\ot N\right),$$
where $M$ is made into a right $\oa$-module by the field goal transform.
\end{corollary}

\begin{corollary}
We have an equivalence of categories,
$$\left\{\textrm{Algebras in $\int_{Ann}\!\!\!\!\cA$}\right\} \simeq \left\{\begin{array}{l}\textrm{Algebras in $\cA$, equipped with an}\\\textrm{algebra homomorphism from $\oa$}\end{array}\right\}.$$
Moreover, given a module category $\cM=A\modu_\cA$, for an algebra $A\in\cA$ equipped with an algebra morphism $\rho:\oa\rightarrow A$, the action of $\int_{Ann}\cA$ on $\cM$ is given by:
\begin{align*}
	A\modu_\cA \bt	\oa\modu_\cA & \longrightarrow  A\modu_\cA\\
	V\bt M &\longmapsto V \ot_{\oa} M
\end{align*}
where $\oa$-acts on $V$ via $\rho$.
\end{corollary}

\begin{corollary}\label{annulus-corollary}
	Let $\cM$ and $\cN$ be a left and a right module category over $\int_{Ann}\cA$ with progenerators $M$ and $N$, respectively. Then there is an equivalence of categories:
	\[
		\cM \underset{\scalebox{0.75}{$\displaystyle{\int_{Ann}\!\!\!\!\cA}$}}{\boxtimes} \cN \simeq (\underline{\End}(M)-\underline{\End}(N))\operatorname{-bimod}_{\oa\modu_\cA}.
	\]
\end{corollary}

\subsection{Braided module structure}
	It follows from the above that if $B$ is an algebra in $\cA$, every algebra morphism $\rho:\oa\to B$ turns $B\modu$ into a module category over $\int_{Ann}\cA$. Hence $\rho$ should correspond to a braided module structure on the category $B\modu_\cA$ in the sense of Definition~\ref{def:BraidedModule}. In this section we make this structure explicit; the construction which follows can be interpreted as a generalisation of~\cite{Donin2003a}.

Let $B$ be an algebra in $\cA$ and $\Gamma$ an automorphism of the action functor
\[
	B\modu_{\cA} \times \cA \rightarrow B\modu_{\cA}.
\]
We begin by constructing a morphism of underlying objects,
\[
\rho_\Gamma:\oa \longrightarrow B,
\]
as follows.  Using the definition of $\oa$ as a co-end, it suffices to define $\rho_\Gamma$ compatibly on each $V^*\ot V$, as below:
\[
	\rho_{\Gamma|_{V^*\ot V}}:  V^* \ot V \xrightarrow{1_B\ot \id^{\ot 2}} B\ot V^* \ot V\xrightarrow{\Gamma_{B,V^*}\ot \id} B \ot V^* \ot V \xrightarrow{\id \ot \ev} B
\]
It is easily checked that this system of maps descends to $\oa$.
\begin{theorem}\label{thm:DM}
	The morphism $\rho_\Gamma$ is an algebra homomorphism if and only if $\Gamma$ satisfies equation~\eqref{eq:DM}:
\begin{align*}
	\Gamma_{M,V\ot W}=\sigma_{V,W}^{-1} \Gamma_{M,W} \sigma_{V,W} \Gamma_{M,V}
\end{align*}
\end{theorem}
\begin{proof}
Let $m_B$ denote the multiplication of $B$. On the one hand, we consider the composition, $m_B\circ (\rho_B \ot \rho_B)$:
\begin{multline}\label{eq:DMLHS}
	(V^* \ot V)\ot (W^* \ot W) \rightarrow (B\ot V^* \ot V)\ot (B\ot W^* \ot W)\\ \xrightarrow{\Gamma_{B,V^*}\ot \id \ot \Gamma_{B,W^*}\ot \id} (B\ot V^* \ot V)\ot (B\ot W^* \ot W)\rightarrow B\ot B\xrightarrow{m_B} B .
\end{multline}
On the other hand, we have the composition, $\rho_B\circ m_{\oa}$:
\begin{multline}\label{eq:DMRHS}
	(V^* \ot V)\ot (W^* \ot W) \xrightarrow{\sigma_{V^*\ot V,W^*}} (W^*\ot V^*)\ot V\ot W\\ \rightarrow B\ot (W^*\ot V^*)\ot V\ot W \xrightarrow{\Gamma_{B,W^*\ot V^*}\ot \id} B\ot (W^*\ot V^*)\ot V\ot W\rightarrow  B.
\end{multline}
In order to simplify the computation, we precompose each side by $\sigma_{V^* \ot V,W^*}^{-1}$. Then the homomorphism equation~\eqref{eq:DMLHS}=\eqref{eq:DMRHS} can be expressed as follows:
\[
	\tik{
		\node[label=above:$B$] at (.5,0.5) {};
		\straight{0.5}{-0.5}
		\ap{0}{0}
		\lmove{-1}{1}\ap{0}{1}\rmove{1}{1}\ap{3}{1}
		\coupon{-1}{2}{1}{$\Gamma$}\straight{1}{2}\coupon{2}{2}{1}{$\Gamma$} \straight{4}{2}
		\straight[-*]{-1}{3}\straight{0}{3}\straight{1}{3}\straight[-*]{2}{3}\straight{3}{3}\straight{4}{3}
		\straight{0}{4} \straight{1}{4} \lmove{2}{4}\lmove{3}{4}
		\straight{0}{5}\negative{1}{5}\straight{3}{5}
		\negative{0}{6}\straight{2}{6}\straight{3}{6}
		\node[label=below:$W^*$] at (0,-7) {};
		\node[label=below:$V^*$] at (1,-7) {};
		\node[label=below:$V$] at (2,-7) {};
		\node[label=below:$W$] at (3,-7) {};
	}
	\quad=\quad
	\tik{
		\node[label=above:$B$] at (0,0) {};
		\straight{0}{0} \ap{2}{0}
		\straight{0}{1} \lmove{1}{1}\ap{2}{1} \rmove{3}{1}
		\coupon{0}{2}{2}{$\Gamma$}\straight{3}{2}\straight{4}{2}
		\straight[-*]{0}{3}\straight{1}{3}\straight{2}{3}\straight{3}{3}\straight{4}{3}
		\node[label=below:$W^*$] at (1,-4) {};
		\node[label=below:$V^*$] at (2,-4) {};
		\node[label=below:$V$] at (3,-4) {};
		\node[label=below:$W$] at (4,-4) {};

	}
\]
Since $\Gamma$ is $B$-linear, $\Gamma_{B,-}$ commutes with the action of $b$ on itself by left multiplication. Thus, beginning with the LHS, we may slide the rightmost instance of $\Gamma$ over the multiplication, at which point the rightmost unit disappears.  We have:
\[
LHS	\quad=\quad
	\tik{
		\node[label=above:$B$] at (0,0) {};
		\straight{0}{0} \ap{1}{0}
		\coupon{0}{1}{1}{$\Gamma$} \straight{2}{1}
		\straight{0}{2}\rmove{1}{2}\rmove{2}{2}
		\straight{0}{3} \ap{1}{3}\rmove{2}{3}\rmove{3}{3}
		\coupon{0}{4}{1}{$\Gamma$}\negative{2}{4} \straight{4}{4}
		\straight[-*]{0}{5}\negative{1}{5} \straight{3}{5}\straight{4}{5}
		\node[label=below:$W^*$] at (1,-6) {};
		\node[label=below:$V^*$] at (2,-6) {};
		\node[label=below:$V$] at (3,-6) {};
		\node[label=below:$W$] at (4,-6) {};

	}
	\quad=\quad
	\tik{
		\node[label=above:$B$] at (0,0) {};
		\straight{0}{0} \ap{2}{0}
		\straight{0}{1} \lmove{1}{1}\ap{2}{1} \rmove{3}{1}
		\coupon{0}{2}{1}{$\Gamma$} \straight{2}{2}\straight{3}{2}\straight{4}{2}
		\straight{0}{3}\positive{1}{3}\straight{3}{3}\straight{4}{3}
		\coupon{0}{4}{1}{$\Gamma$} \straight{2}{4}\straight{3}{4}\straight{4}{4}
		\straight[-*]{0}{5}\negative{1}{5}\straight{3}{5}\straight{4}{5}
		\node[label=below:$W^*$] at (1,-6) {};
		\node[label=below:$V^*$] at (2,-6) {};
		\node[label=below:$V$] at (3,-6) {};
		\node[label=below:$W$] at (4,-6) {};

	}
\]
Clearly, this equals the RHS if, and only if, we have
\[
	\tik{
		\node[label=above:$B$] at (0,-1) {};
		\node[label=above:$W^*$] at (1,-1) {};
		\node[label=above:$V^*$] at (2,-1) {};		
		\straight{0}{1} \straight{1}{1} \straight{2}{1};
		\coupon{0}{2}{1}{$\Gamma$} \straight{2}{2}
		\straight{0}{3}\positive{1}{3}
		\coupon{0}{4}{1}{$\Gamma$} \straight{2}{4}
		\straight{0}{5}\negative{1}{5}
		\node[label=below:$B$] at (0,-6) {};
		\node[label=below:$W^*$] at (1,-6) {};
		\node[label=below:$V^*$] at (2,-6) {};
	}
	\quad =\quad
	\tik{
		\node[label=above:$B$] at (0,0) {};
		\node[label=above:$W^*$] at (1,0) {};
		\node[label=above:$V^*$] at (2,0) {};
		\straight{0}{0}\straight{1}{0}\straight{2}{0}
		\coupon{0}{1}{2}{$\Gamma$}
		\straight{0}{2}\straight{1}{2}\straight{2}{2}
		\node[label=below:$B$] at (0,-3) {};
		\node[label=below:$W^*$] at (1,-3) {};
		\node[label=below:$V^*$] at (2,-3) {};
		
	}
\]
Conjugating the above by $\sigma_{W^*,V^*}$, and replacing $V^*,W^*$ with $V, W$ gives equation \eqref{eq:DM}. Therefore, compositions~\eqref{eq:DMLHS} and~\eqref{eq:DMRHS} coincide if and only if~\eqref{eq:DM} holds.
\end{proof}

We will apply Theorem~\ref{thm:DM} in the following particular case: by definition of $\oa$, there are also canonical maps
\[
V\ot		{}^*V \longrightarrow \oa
\]
where ${}^*V$ is the \emph{right} dual of $V$. For $M\in \oa\modu_{\cA}$ and $V \in \cA$, let $L_\cA \in Aut(M\ot V)$ be the operator defined by
\begin{multline*}
	M\ot V \xrightarrow{\id^{\ot 2}\ot coev_R} M\ot V \ot {}^*V\ot V \xrightarrow{\sigma_{M,V\ot {}^*V} \ot \id} V\ot {}^*V \ot M \ot V \\ \longrightarrow \oa \ot M\ot V \longrightarrow M \ot V
\end{multline*}

where the last map is the action of $\oa$ on $M$.
\begin{proposition}
	Under the identification used in Theorem~\ref{thm:DM}, the map from $\oa$ to itself induced by $L_{\cA}$ is the identity. In particular this is an algebra morphism, hence $L_{\cA}$ satisfies equation~\eqref{eq:DM}.
\end{proposition}
\begin{proof}
We have:
\[
	(\rho_{O_A})_{|V^*\ot V}=	\tik{
		\straight{1.3}{-3.5}
		\begin{scope}[xscale=1.5,xshift=-0.67cm]
		\ap{1}{-3}
		\end{scope}
		\coupon{0}{-2}{1}{$\iota_{V}$} \straight{2}{-2}
		\straight{0}{-1}\straight{1}{-1}\straight{2}{-1}\ap{3}{-1}
		\straight{0}{0} \positive{1}{0}\straight{3}{0}\straight{4}{0}
		\positive{0}{1}\up{2}{1} \straight{4}{1}
		\straight[-*]{0}{2}\straight{1}{2}\straight{4}{2}
		\node[label=below:$\un_{\cA}$] at (0,-3) {};
		\node[label=below:$V^*$] at (1,-3) {};
		\node[label=below:$V$] at (4,-3) {};
	}=\iota_{V}
\]
\end{proof}
\begin{corollary}
Let $B$ be an algebra in $\cA$. If $\Gamma$ a braided module structure on $B\modu$, then $\rho_\Gamma$ is an algebra morphism $\oa\rightarrow B$. Conversely, if $\rho$ is an algebra morphism $\oa \rightarrow B$, then using $\rho$ $\oa$ acts on any $B$-modules $M$,  hence $L_\cA$ acts on $M \ot V$ for $V\in \cA$ and this turns $B\modu$ into a braided module category.
\end{corollary}

\section{Closed surfaces, markings and quantum Hamiltonian reduction}\label{sec:closed}
We now close up punctured surfaces to describe quantum character varieties for closed, possibly marked, surfaces.  For simplicity, we will begin with the unmarked situation.

Fixing an inclusion $D^2\to S$, and the corresponding boundary inclusions of the annulus into $S^\circ=S\setminus D^2$ and $D^2$ as boundary annuli, excision gives a canonical equivalence,
$$\int_{S} \cA\simeq \int_{S^\circ}\cA \underset{\scalebox{0.75}{$\displaystyle{\int_{Ann}\!\!\!\!\!\!\!\cA}$}}{\bt} \cA.$$
Here $\cA=\int_{D^2}\cA$ is regarded as a braided module category over itself.

We now proceed to compute these invariants explicitly via quantum Hamiltonian reduction.  Recall that the algebra $A_{S^\circ}$ is identified with $\act^R_{\cO_{S^\circ}}(\cO_{S^\circ})\in \cA$, where $\cA$ acts on $\int_{S^\circ}\cA$ via insertion at some interval in the boundary annulus of $\cA$.  We have $\cA$-module category equivalences,
$$\int_{S^\circ}\cA\simeq F_{bd}^*\int_{S^\circ}\cA,$$
where $F_{bd}:\cA\to\int_{Ann}\cA$ is the tensor functor from Definition \ref{annulus-tensor-functor-def}, induced by the chain of inclusions depicted in Figure \ref{nested-actions}.

\begin{figure}[h]
\includegraphics[height=1in]{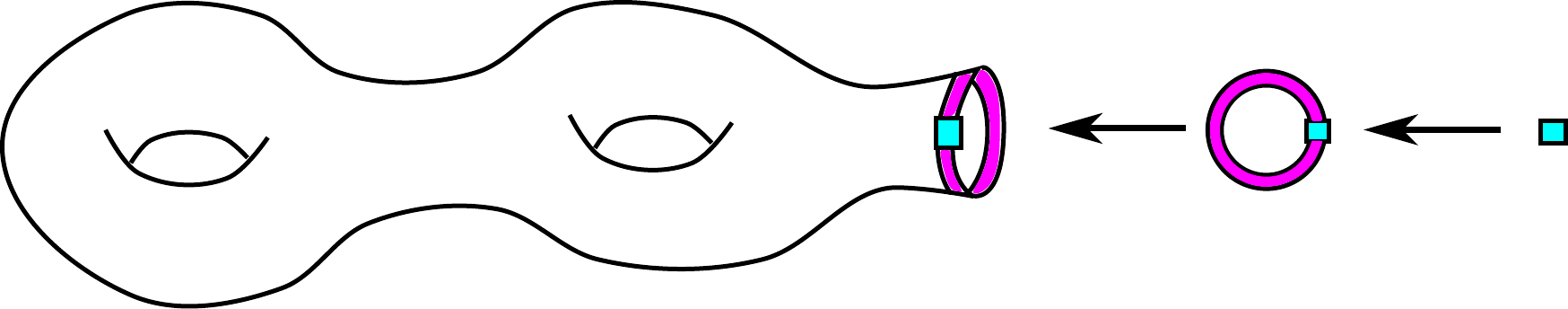}
\caption{The action of $\cA$ on $\int_{S^\circ}\cA$ is obtained by pulling back the $\int_{Ann}\!\cA$-action, through the tensor functor $F_{bd}:\cA\to\int_{Ann}\cA$.}\label{nested-actions}
\end{figure}

This means we are precisely in the situation of Section \ref{annulus-category-section}.  We have canonical quantum moment maps,
$$\mu: \oa\to A_{S^\circ},\qquad \epsilon:\oa\to \mathbf{1}_\cA,$$
which are $\cA$-algebra homomorphisms realizing the braided module structures on $$\int_{S^\circ}\cA\simeq A_{S^\circ}\modu_{\cA}, \qquad \int_{D^2}\cA \simeq \cA\simeq \mathbf{1}_\cA\modu_{\cA},$$
respectively, and moreover we have equivalences of $\int_{Ann}\cA$-module categories,
$$\int_{S^\circ}\cA \simeq A_{S^\circ}\modu_{\oa\modu_\cA}, \qquad \int_{D^2}\cA \simeq \mathbf{1}_\cA\modu_{\oa\modu_\cA}$$
between the factorization homology of $S^\circ$ (resp, $D^2$), as module categories for the annulus category by stacking, and the categories of $A_{S^\circ}$-modules in $\cA$ (resp, $\mathbf{1}_\cA$-modules in $\cA$), equipped in each case with compatible actions of $\oa$.

\begin{remark} The map $\mu$ is a generalization of the ``quantum moment maps" studied in~\cite{Jordan2014,Varagnolo2010}; it quantizes the monodromy map,
$$\uch_G(S\setminus D^2) \to \frac{G}{G}.$$
\end{remark}

As an application of Corollary \ref{annulus-corollary}, we obtain:

\begin{theorem} The category attached to a closed surface is given as $(A_{S^\circ},\mathbf{1}_\cA)$-bimodules in the annulus category,
$$\int_S\cA\simeq (A_{S^\circ}\bimodu 1_\cA)_{\oa\modu_\cA}.$$
\end{theorem}

\subsection{Marked points}

Let $X\subset S$ denote a finite set of points.  Let us fix braided $\cA$-module categories $\cM_i$ attached to each points $x_i\in X$, and let $\overline{X} = D_{x_1}\cup \cdots \cup D_{x_r}$ denote a tubular neighborhood of $X$, consisting of the disjoint union of some discs $D_{x_i}$ containing each point $x_i$, and let $S^\circ = S\setminus \overline{X}$.  This data defines an invariant,
$$(S,X)\mapsto \int_{(S,X)} \!\!\!\!\!\!\left(\cA,\{\cM_1,\ldots,\cM_r\}\right)$$
of surfaces with marked point $x$ labeled by $\cM$. Applying excision gives an equivalence:
$$\int_{(S,X)} \!\!\!\!\!\!\left(\cA,\{\cM_1,\ldots,\cM_r\}\right)\simeq\int_{S^\circ}\!\!\!\!\!\cA \underset{\scalebox{0.75}{$\displaystyle{\left(\int_{Ann}\!\!\!\!\!\!\cA\right)}$}^{\bt I}}{\scalebox{1.5}{$\bt$}}\left(\cM_1\bt\cdots\bt \cM_r\right).$$

Now let us assume furthermore that each $\cM_i$ is given as the category of modules for an algebra $A_i\in\cA$, i.e. that
$$\cM_i=A_i\modu_\cA.$$
Giving such a presentation is equivalent to giving an $\cA$-progenerator $M_i\in\cM_i$ as an $\cA$-module category, by taking $A_i=\underline{\End}(M_i)$.  It follows from Theorem~\ref{quantum-moment-map-theorem} that each $A_i$ canonically receives a quantum moment map,
$$ \mu_i:\oa\to A_i,$$
such that the $\int_{Ann}\cA$-module action is identified with the relative tensor product over $\mu_i$.

Applying Corollary \ref{annulus-corollary}, we obtain:

\begin{theorem}\label{marked-surface-category-theorem} The factorization homology of the marked surface $(S,X)$, with braided module categories $\cM_i=\modul$-$A_i$ attached to each $x_i$ is equivalent to the category,
$$ \int_{(S,X)} (\cA,\{\cM_1,\ldots,\cM_r\})\simeq   (A_{S\setminus \overline{X}}\bimodu (A_1\bt\cdots\bt A_r)_{\oa^{\bt r}\modu_\cA},$$
 of bimodules in $\oa^{\ot r}\modu_\cA$ for the pair of algebras $A_{S\setminus \sqcup D^2_i}$ and $A_1\bt\cdots \bt A_r$.
\end{theorem}

\subsection{The functor of global sections}

Let us now turn to computing the endomorphism algebra of the distinguished object $\cO_{\cA,S}\in\int_S\cA$.  To this end, we recall first of all that we have an isomorphism,
\begin{equation}\label{eq:S-decomp}\cO_{\cA,S} \cong \cO_{\cA,S^\circ} \underset{\scalebox{0.75}{$\displaystyle{\int_{Ann}\!\!\!\!\!\!\!\cA}$}}{\bt}\cO_{\cA,D^2}.\end{equation}

More generally, for any braided module category $\cM$, and any $N\in\cM$, we have,
\begin{equation}\label{eq:S-decomp-general}i_{x*}(N) \cong \cO_{\cA,S^\circ} \underset{\scalebox{0.75}{$\displaystyle{\int_{Ann}\!\!\!\!\!\!\!\cA}$}}{\bt}N.\end{equation}

\begin{theorem}\label{HamRedAlg}
We have an isomorphism of algebras,
$$\End(\cO_{\cA,S}) \cong \Hom_{\cA}(\mathbf{1}_\cA,A_{S\setminus D^2}\underset{\oa}{\ot} \mathbf{1}_\cA),$$
between the endomorphism algebra of the distinguished object and the Hamiltonian reduction algebra.

More generally, if $S$ is a surface with a point marked by a braided module category $\cM$, then, for any $N\in\cM$, we have an isomorphism of algebras:
$$\End(i_{x*}(N)) \cong \Hom_{\cA}(\mathbf{1}_\cA,A_{S\setminus D^2}\underset{\oa}{\ot}\underline{\End}(N)),$$
where the relative tensor product is taken over the canonical homomorphism $\rho: \oa\to N$ given by Proposition \ref{momentmaps}.
\end{theorem}
\begin{proof}
The proof is based on a concrete description of Hom spaces between pure tensor products of objects in relative tensor product categories, which was proved in~\cite{Douglas2013}.  We recall:

\begin{proposition}[\cite{Douglas2013}]
Given a right module category $\cM$, and a left module category $\cN$ over a rigid tensor category $\cC$, and $m,m'\in \cM, n,n'\in \cN$, we have an isomorphism,
$$ \Hom_{\cM\underset{\cC}{\bt}\cN}(m \underset{\cC}{\bt} n, m' \underset{\cC}{\bt} n') \cong \Hom_\cC(\mathbf{1}_\cC, \underline{\Hom}(m,m')\ot \underline{\Hom}(n,n')).$$
\end{proposition}
Combining the above proposition with the tensor product decomposition of equation \eqref{eq:S-decomp}, we have an isomorphism,
\begin{align*}\End_{\int_S\cA}(i_*(N)) &\cong \End_{\int_S\cA}\left(\cO_{\cA,S\setminus D^2}\underset{\scalebox{0.75}{$\displaystyle{\int_{Ann}\!\!\!\!\!\!\!\cA}$}}{\bt} N\right)\\
&\cong \Hom_{\int_{Ann}\cA}\left(\oa,\underline{\End}_{\int_{Ann}\cA}(\cO_{\cA,S\setminus D^2})\odot \underline{\End}_{\int_{Ann}\cA}(N)\right)\\
&\cong \Hom_{\int_{Ann}\cA}\left(F_{bd}(\mathbf{1}_\cA),\underline{\End}_{\int_{Ann}\cA}(\cO_{\cA,S\setminus D^2})\odot \underline{\End}_{\int_{Ann}\cA}(N)\right)\\
&\cong \Hom_{\cA}\left(\mathbf{1}_\cA,F_{bd}^R\left(\underline{\End}_{\int_{Ann}\cA}(\cO_{\cA,S\setminus D^2})\odot \underline{\End}_{\int_{Ann}\cA}(N)\right)\right)\\
&\cong \Hom_{\cA}\left(\mathbf{1}_\cA, A_{S\setminus D^2}\underset{\oa}{\ot} \underline{\End}_{\cA}(N)\right),
\end{align*}
as claimed.
\end{proof}

\begin{remark} In the case $\cA=\Rep G$, the distinguished object can be identified with the structure sheaf of the character stack.  Hence the functor $\Gamma = \Hom(\cO_{\cA,S},-)$ can be viewed as a ``global sections functor" on the quantized character stack $\int_S \Repq  G$.  We note that the procedure prescribed in Theorem \ref{HamRedAlg}, of tensoring with the trivial module along the quantum moment map, and then taking invariants, is precisely the procedure of quantum Hamiltonian reduction.
\end{remark}


\end{document}